\documentclass{amsart}

\usepackage{amssymb}
\usepackage{tikz}
\usetikzlibrary{positioning,decorations.pathreplacing}
\usepackage{mathrsfs}

\usepackage[bookmarks]{hyperref}

\newtheorem{theorem}{Theorem}[section]
\newtheorem{lemma}[theorem]{Lemma}
\newtheorem{proposition}[theorem]{Proposition}
\newtheorem{corollary}[theorem]{Corollary}
\newtheorem{conjecture}[theorem]{Conjecture}
\newtheorem{problem}[theorem]{Problem}

\theoremstyle{definition}
\newtheorem{definition}[theorem]{Definition}

\theoremstyle{remark}
\newtheorem{claim}{Claim}

\let\al\mathbf
\let\var\mathcal
\let\clo\mathscr
\let\rel\mathbb
\let\tup\mathbf
\let\join\vee
\let\meet\wedge
\let\equals\approx

\let\Union\bigcup
\newcommand{\notni}{\not\mkern1mu\ni}
\newcommand{\idclo}[1]{\clo{#1}^{\text{id}}}
\newcommand{\freeclo}[1]{{#1^{\rel F}}}

\DeclareMathOperator{\Cg}{Cg}
\DeclareMathOperator{\Cross}{Cross}

\DeclareMathOperator{\HSP}{\mathsf{HSP}}

\author[J. Opr\v{s}al]{Jakub Opr\v{s}al}
\email{jakub.oprsal@tu-dresden.de}
\address{Department of Algebra, Faculty of Mathematics and Physics,
  Charles University in~Prague, Sokolovsk\'a 83, 186 75 Praha, Czech
  Republic, and
  Institut f\"ur Algebra, Technische Universit\"at Dresden,
  01062 Dresden, Germany}

\title[Taylor's modularity conjecture and related problems]
  {Taylor's modularity conjecture and related problems for~idempotent varieties}
\date{\today}


\thanks{Supported by Czech Science Foundation, GA\v{C}R 13-01832S; and Charles
  University in Prague, project GAUK-558313.  The author has also received
  funding from the European Research Council (Grant Agreement no.\ 681988,
  CSP-Infinity) and the National Science Centre Poland under grant no.\
  UMO-2014/13/B/ST6/01812}

\hypersetup{
  pdfauthor={Jakub Opr\v{s}al},
  pdftitle={Taylor's modularity conjecture and related problems for idempotent
    varieties},
  pdfborder=0 0 0,
}

\begin{document}

\maketitle

\begin{abstract}
We provide a~partial result on Taylor's modularity conjecture, and several
related problems. Namely, we show that the interpretability join of two
idempotent varieties that are not congruence modular is not congruence modular
either, and we prove an analogue for idempotent varieties with a~cube term.
Also, similar results are proved for linear varieties and the properties of
congruence modularity, having a~cube term, congruence $n$-permutability for
a~fixed $n$, and satisfying a~non-trivial congruence identity.
\end{abstract}

\section{Introduction}

An~\emph{interpretation} $\iota$ of a~variety $\var V$ in a~variety $\var W$ is
a~mapping that maps basic operations of $\var V$ to terms of $\var W$ of the
same arity such that for every algebra $\al A\in \var W$, the algebra $(A,
(\iota(f)^{\al A})_{f\in \sigma})$ (where $\sigma$ is the signature of $\var V$)
is an algebra in $\var V$.  We say that a~variety $\var V$ is interpretable in
a~variety $\var W$ if there exist an~interpretation of $\var V$ in $\var W$. The
\emph{lattice of interpretability types of varieties} (see
\cite{neumann74,garcia.taylor84}) is then constructed by quasi-ordering all
varieties by interpretability, and factoring out varieties that are
interpretable in each other. This gives a~partially ordered class such that
every set has a~join and a~meet.  The lattice of interpretability types of
varieties is a~suitable object for expressing properties of Mal'cev
conditions~(for a~formal definition see \cite{taylor73}): The varieties that
satisfy a~given Mal'cev condition form a~filter in this lattice, thus, e.g.\
implications among Mal'cev conditions translate into inclusions among the
corresponding filters. 

In this paper, we contribute to the line of research whose aim is to understand
which of the important Mal'cev conditions are indecomposable in the following
strong sense: if two sets of identities in disjoint languages together imply
the Mal'cev condition, then one of the sets already do. An equivalent
formulation using the interpretability lattice is especially simple: which of
the important Mal'cev conditions determine a~prime filter?
Some of the Mal'cev conditions with this property have been described in the
monograph by Garcia and Taylor \cite{garcia.taylor84}, e.g.\  having a~cyclic
term of given prime arity. Garcia and Taylor conjectured that the filter of
congruence permutable varieties and the filter of congruence modular varieties
are prime. For congruence permutability, this was confirmed by
Tschantz~\cite{tschantz96}. Unfortunately, this proof has never been published.
The congruence modular case is still open:

\begin{conjecture}[Taylor's modularity conjecture]
  The filter of congruence modular varieties is prime, that is, if $\var V$ and
  $\var W$ are two varieties such that $\var V \join \var W$ is congruence
  modular, then either $\var V$ or $\var W$ is congruence modular.
\end{conjecture}

In \cite{bentz.sequeira14}, Bentz and Sequeira proved that this is true if $\var
V$ and $\var W$ are idempotent varieties that can be defined by linear
identities (such varieties are called linear idempotent varieties), and later in
\cite{barto.oprsal.ea15}, Barto, Pinsker, and the author generalized their
result to linear varieties that do not need to be idempotent. In this paper we
generalize Bentz and Sequeira's result in a~different direction.

\begin{theorem} \label{thm:taylors-conjecture-introduction}
  If $\var V$, $\var W$ are two idempotent varieties such that $\var V \join
  \var W$ is congruence modular then either $\var V$ or $\var W$ is congruence
  modular.
\end{theorem}

Several similar partial results on primeness of some Mal'cev filters have been
obtained before. Bentz and Sequeira in \cite{bentz.sequeira14} also proved for
two linear idempotent varieties: if the join of the two varieties is congruence
$n$-permutable for some $n$, then so is one of the two varieties; and
similarly if their join satisfies a~non-trivial congruence identities, then so
does one of the two varieties. Stronger versions of these results also follow
from the work of Valeriote and Willard \cite{valeriote.willard14}, who proved
that every idempotent variety that is not $n$-permutable for any $n$ is
interpretable in the variety of distributive lattices, and the work of Kearnes
and Kiss \cite{kearnes.kiss13}, who proved that any idempotent variety which
does not satisfy a~non-trivial congruence identity is interpretable in the
variety of semilattices. Recently, a~similar result has been obtained by Kearnes
and Szendrei \cite{kearnes.szendrei15} for the filter of varieties having a~cube
term.

\begin{theorem} \label{thm:cube-terms-introduction}
  Suppose that $n\geq 2$, and let $\var V$ and $\var W$ be two idempotent
  varieties such that $\var V \join \var W$ has an~$n$-cube term.  Then so does
  either $\var V$ or~$\var W$.
\end{theorem}

We provide an alternative proof of this result using the fact (obtained in
\cite{kearnes.szendrei15} and recently also by \cite{mckenzie.moore17}) that
idempotent varieties that do not have a~cube term contain an algebra with a~cube
term blocker.

All of the filters mentioned so far share the following property: the varieties
from their complements have a~so-called strong coloring of their terms by
a~finite relational structure that depends on the particular filter. The precise
definition is given in Section~\ref{sec:3.2}. The notion is a~reformulation of
colorings described in \cite{barto.oprsal.ea15}, and a~generalization of
compatibility with projections introduced in \cite{sequeira01}. In the present
manuscript, we describe these coloring structures.
These characterizations by the means of colorings allow us to give analogous
results for linear varieties. Moreover, we are also able to connect these
results with their analogues for idempotent varieties (when those are available).

\begin{theorem} \label{thm:overview-introduction} \label{thm:1.4}
  Let $\var V$ and $\var W$ be two varieties such that each of them is either
  linear or idempotent.
  \begin{itemize}
  \item[(i)] If $\var V\join \var W$ is congruence modular, then so is either
    $\var V$ or $\var W$;
  \item[(ii)] if $\var V\join \var W$ is congruence $k$-permutable for some $k$,
    then so is either $\var V$ or $\var W$;
  \item[(iii)] if $\var V\join \var W$ satisfies a~non-trivial congruence
    identity, then so does either $\var V$ or $\var W$;
  \item[(iv)] for all $n\geq 2$: if $\var V \in \var W$ has an~$n$-cube term,
    then so does either $\var V$ or $\var W$.
  \end{itemize}
\end{theorem}

\begin{theorem} \label{thm:permutability-introduction} \label{thm:1.5}
  Let $n\geq 2$. If $\var V$ and $\var W$ are two linear varieties such that
  $\var V \join \var W$ is congruence $n$-permutable, then so is either $\var
  V$ or $\var W$.
\end{theorem}

However, we are not able to answer the following. (Note that the case $n=2$ has
been resolved in \cite{tschantz96} as well as Theorem~\ref{thm:1.4}(iv).)

\begin{problem} Given $n > 2$ and two idempotent varieties $\var V$ and $\var
W$ such that $\var V \join \var W$ is $n$-permutable. Is it always true that
either $\var V$ or $\var W$ is $n$-permutable?
\end{problem}
 \section{Varieties, clones, and relational structures}

Before we get to prove the results of this paper, we would like to recall
another constructions of a~class-size lattices that is equivalent to the
lattice of interpretability types of varieties. That is the lattice of
(homomorphism classes of) clones ordered by an existence of a~clone
homomorphism: We start with a~preorder on clones defined by the existence of a~clone
homomorphism between the two clones, and follow by factoring out the
homomorphically equivalent clones.  The equivalence of the two construction is
implicitly given by Birkhoff's theorem.
There is also a~third construction that lacks being completely equivalent to
the two: taking relational structures and pp-interpretability between them.
Those are connected to clones by the Galois correspondence between relations
and function clones on a~fixed (finite) set. Even though this construction
of the lattice of interpretability types lacks to be equivalent in general, we
will often describe some function clones as clones of polymorphisms of
a~relational structure.

\subsection{Function clones} \label{sec:2.1}
A~function clone (or just a clone) is a~set of operations $\clo A$ on a~fixed
set $A$ that contains projections and is closed under composition. We will
always use the same letter in italic font to denote the underlying set of the
clone.
By a~\emph{clone homomorphism} from a~function clone $\clo A$ to a~function
clone $\clo B$ we mean a~mapping $\xi\colon \clo A \to \clo B$ that preserves
composition and projections.

The correspondence between clones and varieties is given by the following: Any
clone $\clo A$ can be naturally understood as an~algebra $\al A = (A, (f)_{f\in
\clo A})$ having the signature $\clo A$; the corresponding variety is then the
variety generated by $\al A$. The same variety can be also understood as the
variety of all actions of $\clo A$ on sets (each action of $\clo A$ gives an
algebra in $\HSP \al A$). We will call this variety \emph{the variety of
actions of $\clo A$}.
For the other way, to obtain a~clone corresponding to a~variety $\var V$ we
take the clone of term functions of its countably-generated free algebra $\clo
F$, or in fact a~clone of term functions of any generator of $\var V$ (any two
such clones are isomorphic). When we refer to a~\emph{clone of a~variety} we
mean the function clone $\clo F$ above.

\subsection{Relational structures and pp-interpretations}
We say that a~relation $R$ is \emph{pp-definable} in the structure $\rel B$ if
there is a~primitive positive formula (using only conjunction and existential
quantifiers) $\psi$ such that $(b_1,\dots,b_k) \in R$ if and only if $\rel B
\models \psi(b_1,\dots,b_n)$, and we say that a~structure $\rel A$ is
\emph{pp-interpretable} in $\rel B$ if there is a~pp-definable relation $R$ and
a~surjective mapping $f\colon R \to A$ such that the kernel of $f$ and the
relations $f^{-1}(S^{\rel A})$, where $S$ is a~relation $S$ of $\rel A$, are
pp-definable in $\rel B$.

The following relation between pp-interpretations and clone homomorphisms for
finite underlying sets is a~consequence of the Galois correspondence between
function clones and pp-closed systems of relations (see e.g.\ \cite{geiger68})
and Birkhoff's HSP theorem.

\begin{theorem} \label{thm:structures-clones}
Let $\rel A$, $\rel B$ be finite relational structures and $\clo A$\!, $\clo B$,
respective, their polymorphism clones.
Then the following are equivalent.
\begin{itemize}
  \item[(i)] $\rel B$ is pp-interpretable in $\rel A$.
  \item[(ii)] There exists a~clone homomorphism from $\clo A$ into $\clo B$.
  \qed
\end{itemize}
\end{theorem}

For infinite structures, there are a~few issues that can be resolved by
restricting to countable $\omega$-categorical structures and considering
uniformly continuous clone homomorphisms, or relaxing a~definition of relational
structure to allow infinitary relations and allow infinite `pp-definitions' (see
\cite{bodirsky.pinsker15} and \cite{romov77}).  Nevertheless, the implication
(i) $\to$ (ii) is valid in general, i.e., if the structures $\rel A$ and $\rel
B$ are infinite and not necessarily $\omega$-categorical.

\subsection{Notation}
We use letters in italic for underlying sets of clones, algebras, and relational
structures that are denoted by the same letters. That means the symbol $A$ is
used to denote an~underlying set of: an algebra $\al A$, a~clone $\clo A$, and
a~relational structure $\rel A$. We will keep this consistent, i.e., if we
denote two structures (algebraic and relational) by the same letter, they have
the same underlying set.
Moreover, $\clo A$ denotes the clone of polymorphisms of the relational
structure $\rel A$. We will also keep some consistence between algebras,
relational structures, and clones: all operations of an~algebra $\al A$ will be
compatible with relations of $\rel A$ and belong to the clone $\clo A$.

We also say that an algebra $\al A$ is \emph{compatible} with a~relational
structure $\rel B$, if the above is the case, that is, they share a~universe ($A
= B$) and each operation of $\al A$ is compatible with every relation of
$\rel B$.
 \section{Overview of the method}

Each of the Mal'cev filters given by one of the conditions mentioned in
Theorems~\ref{thm:1.4} and~\ref{thm:1.5} can be described as the class of
varieties that do not contain an algebra that has some compatible relations with
a~special property (e.g.\ congruence modular are those varieties that do not
contain an algebra having three congruences that do not satisfy the modular law,
finitely generated varieties with cube terms are those that do not contain an
algebra with `too many subpowers', etc.). In order to prove that these filters
are prime, we have to prove that for any two varieties $\var V$, $\var W$ that
contain such `ugly' algebras, also $\var V \join \var W$ contains such an
algebra. More precisely, Taylor's modularity conjecture in fact states that for
any two varieties $\var V$, $\var W$ that are not congruence modular, there
exists an algebra $\al A$ in $\var V \join \var W$ that has congruences
$\alpha$, $\beta$, and $\gamma$ that do not satisfy the modularity law. Such an
algebra has two natural reducts $\al A_1 \in \var V$, $\al A_2 \in \var W$ which
are obtained by taking only those basic operations of $\al A$ that belong to the
signature of the respective variety. Both these algebras share a~universe and
the three congruences $\alpha$, $\beta$, and $\gamma$. Therefore, Taylor's
conjecture states that for any two congruence non-modular varieties, there exist
a~set $A$ and  equivalence relations $\alpha$, $\beta$, $\gamma$ on $A$ such
that both varieties contain an algebra with the universe $A$ and congruences
$\alpha$, $\beta$, and $\gamma$.  In other words (see Lemma~\ref{lem:2.1}), both
varieties are interpretable in the variety of actions of the polymorphism clone
$\clo A$ of the relational structure $(A; \alpha, \beta, \gamma)$.
Although it would be convenient, we cannot find a~single relational structure
$\rel B$ such that a~variety is not congruence modular if and only if it is
interpretable in the variety of actions of~$\clo B$. This is due to the fact,
that congruence non-modular varieties can omit algebras of size smaller then any
fixed cardinal. Instead, we will find a~chain of clones $\clo P_\kappa$ indexed
by cardinals, that are polymorphism clones of relational structures of
increasing sizes, such that every non-modular idempotent variety is
interpretable in the variety of actions of $\clo P_\kappa$ for every big enough
cardinal $\kappa$. The same method is also applied for the other filters in the
interpretability lattice with the only difference being use of different
relational structures.

Contrary to the general case, in the case that a~variety $\var V$ is linear and
not congruence modular, we are able to prove that it is interpretable in the
variety of actions of the clone $\clo P_0$, that is defined as the clone of
polymorphisms of $\rel P_0 = (\{ 0,1,2,3 \}, \alpha, \beta, \gamma )$ where
$\alpha$, $\beta$, and $\gamma$ are equivalences defined by partitions $01|23$,
$03|12$, and $0|12|3$, respectively. This is due to the fact that non-modular
varieties can be described as those varieties whose terms are colorable by $\rel
P_0$ (see Proposition~\ref{prop:coloring-for-modularity}). This turns out to be,
in the case of linear varieties, equivalent to any of the conditions in
Lemma~\ref{lem:2.1} for the relational structure $\rel P_0$ (see
Lemma~\ref{lem:linear-varieties}). Moreover, we will also show that there is
a~clone homomorphism from $\clo P_0$ to $\clo P_\kappa$ for any infinite
cardinal $\kappa$, connecting the results for linear varieties and idempotent
varieties. 

The rest of this section is dedicated to describing several general results
to be used in the following sections.

\subsection{Varieties of actions of polymorphism clones}

Our method is based on the following lemma that describes the class of varieties
that are interpretable in the~variety of actions of the polymorphism clone of
a~relational structure. Here, we denote the signature of a~variety $\var V$ by
$\sigma(\var V)$.

\begin{lemma} \label{lem:2.1}
  Let $\var V$ be a~variety, $\clo F$ the clone of its countably generated free
  algebra, $\rel B$ be a~relational structure.  The
  following are equivalent:
  \begin{enumerate}
    \item There is a~clone homomorphism from $\clo F$ to $\clo B$;
    \item $\var V$ is interpretable in the variety of actions of $\clo B$;
    \item $\var V$ contains an algebra compatible with $\rel B$.
  \end{enumerate}
\end{lemma}

\begin{proof}
  (1) $\to$ (2): Let $\al F$ denote the countable generated free algebra, and
  let $\xi\colon \clo F \to \clo B$ be a~clone homomorphism. We
  will use it to define an interpretation $\iota$ from $\var V$ to the variety
  of action of $\clo B$. Note that the signature of the variety of actions of
  $\clo B$ is the set $\clo B$. For a~$k$-ary symbol $f$ in the signature of
  $\var V$, we define
  \(
    \iota(f) = \xi(f^{\al F})
  \).
  We get that $(B; \iota(f)_{f\in \sigma(\var V)}) \in \var V$ since $\xi$ preserves
  identities, and consequently, since $(B; (f)_{f\in \clo B})$ generates the
  variety of actions of $\clo B$, $\iota$ is an interpretation.

  (2) $\to$ (3): Suppose that $\iota$ is an interpretation of $\var V$ in the
  variety $\var W$ of actions of $\clo B$.
  Note that, since $\clo B$ is closed under compositions, we can without loss of
  generality suppose that $\iota$ maps basic operations of $\var V$ to basic
  operations of $\var W$ (i.e., the elements of $\clo B$).
  From the definition of an interpretation, we have that for any action of $\clo
  B$ on a~set $C$, $(C,(\iota(f))_{f\in \sigma(\var V)})$ is an algebra in $\var
  V$ (here $\iota(f)$ is understood as the action of the~element $\iota (f) \in
  \clo B$). Considering in particular the natural action of $\clo B$ on $B$, we
  claim that $\al A = (B,(\iota(f))_{f\in \sigma(\var V)})$ has the required
  properties.  Indeed, it has the right universe, and moreover for any $f\in
  \sigma(\var V)$, the function $\iota(f)$ is a~polymorphism of~$\rel B$.

  (3) $\to$ (1): Let $\al A$ be the algebra satisfying the condition $(3)$, and
  let $\al F$ be the countably generated free algebra in $\var V$. Define
  $\xi\colon \clo F \to \clo B$ by putting $\xi(f^{\al F}) = f^{\al A}$ for
  every term $f$ of $\var V$. This mapping is well-defined since $\al A$
  satisfies all identities that are satisfied in $\al F$ and every term
  operation of $\al A$ is a~polymorphism of $\rel B$. It is also a~clone
  homomorphism since it preserves projections and composition.
\end{proof}

Also from the above, we get the following property of the interpretability join
of two varieties.

\begin{lemma} \label{lem:joins}
  Let $\rel B$ be a~relational structure, and let $\var V$ and $\var W$ be two
  varieties such that both contain an algebra compatible with $\rel B$. Then
  also the interpretability join $\var V \join \var W$ contains an algebra
  compatible with $\rel B$.
\end{lemma}

\begin{proof}
  From the previous lemma, we get that both $\var V$ and $\var W$ are
  interpretable in the variety of actions of $\clo B$, therefore also $\var V
  \join \var W$ is by the definition of interpretability join. Again by the
  previous lemma, this implies that $\var V\join \var W$ contain an algebra
  compatible with $\rel B$.

  Alternatively, we can construct such an algebra directly: Suppose that
  $\al C \in \var V$ and $\al D \in \var W$ are compatible with
  $\rel B$ (in particular $C = D = B$). The algebra $\al B$ in $\var V\join \var W$
  compatible with $\rel B$ is obtained by putting the structures of these two
  algebras on top of each other, i.e., as
  $\al B = (B; (f^{\al C})_{f\in \sigma(\var V)}, (g^{\al D})_{g\in \sigma(\var
  W)})$.
\end{proof}

\subsection{Colorings of terms by relational structures} \label{sec:3.2}

Usually we are unable to find a~single relational structure $\rel B$ such that
variety falls in the complement of a filter if and only if it is interpretable
in the variety of actions of $\clo B$.
Nevertheless, we can relax from taking~clone homomorphism, corresponding to
interpretation of varieties, to h1 clone homomorphisms. An h1 clone homomorphism
is a~mapping between two clones that preserves identities of height~1 in a
similar way as clone homomorphism preserves all identities
\cite[Definition 5.1]{barto.oprsal.ea15}.
Existence of an~h1 clone homomorphism from the clone of a~variety $\var V$ to the
polymorphism clone of a~relational structure $\rel B$ is described by an
existence of a~coloring described below.
As noted in \cite{barto.oprsal.ea15}, h1 clone homomorphisms encapsulate not
only the usual algebraical (primitive positive) construction, but also
homomorphic equivalence of relational structures. This said, coloring is defined
as a~homomorphism from a~certain `freely generated' relational structure given
by $\var V$ and $\rel B$ to the relational structure $\rel B$. There will always
be a~natural homomorphism from $\rel B$ to this freely generated structure.

\begin{definition} \label{def:free-structure}
  Given a~variety $\var V$ and $\rel B$ a~relational structure.  We define
  a~\emph{free structure generated by $\rel B$} to be a~relational structure
  $\rel F$ of the same signature as $\rel B$ whose underlying set is the
  universe of the free $B$-generated algebra $\al F$, and a~relation $R^{\rel
  F}$ is the smallest compatible relation on $F$ containing $R^{\rel B}$.
\end{definition}

As for a~free algebra, the elements of $\rel F$ are represented by terms of
$\var V$ over the set $B$ as their set of variables. Therefore, it is sensible
to denote the elements of the copy of $B$ in $F$ as
variables $x_b$, $b\in B$. Another way is to define $\al F$ as
being generated by the set $\{ x_b : b\in B \}$ rather then $B$ itself, and
$R^{\rel F}$ being the relations generated by $\{ (x_{b_1},\dots,x_{b_k}) :
(b_1,\dots,b_k) \in R^{\rel B} \}$. It is clear that both relational structures
are isomorphic, moreover they are also isomorphic as (mixed) structures
with both relations $R^{\rel F}$ and operations $f^{\al F}$.

\begin{definition} \label{def:coloring}
  A~\emph{coloring of terms} of a~variety $\var V$ by $\rel B$ (or \emph{$\rel
  B$-coloring of terms}) is any~homomorphism from the free structure
  generated by $\rel B$ to $\rel B$. A~coloring is \emph{strong} if its
  restriction to $B$ is an indentity mapping (maps $x_b$ to $b$).
\end{definition}

The following is a~consequence of the theory developed
in~\cite{barto.oprsal.ea15}; it follows from Propositions 7.4 and 7.5 of the
mentioned paper. Nevertheless, we present a~variant of the proof of equivalence
of (3) and (5) for completeness.

\begin{lemma} \label{lem:linear-varieties}
  Let $\var V$ be a~linear variety, $\clo F$ its clone, and $\rel B$
  a~relational structure.
  The following are equivalent:
  \begin{enumerate}
    \item There is a~clone homomorphism from $\clo F$ to $\clo B$;
    \item there is a~strong h1 clone homomorphism from $\clo F$ to $\clo B$;
    \item the terms of $\var V$ are strongly $\rel B$-colorable;
    \item $\var V$ is interpretable in the variety of actions of $\clo B$;
    \item $\var V$ contains an algebra compatible with $\rel B$.
  \end{enumerate}
\end{lemma}

\begin{proof} The equivalence of (1), (4), and (5) follows from
Lemma~\ref{lem:2.1}.
The implication $(1) \to (2)$ is trivial, and its converse is given by
\cite[Proposition 7.5]{barto.oprsal.ea15}.
To finish the proof, we show that (3) and (5) are equivalent.
We denote $\al F_B$, the free algebra in $\var V$ generated by
$B$, and $\rel F_B$, the free structure generated by $\rel B$. 

$(3) \to (5)$: A~strong coloring
is a~mapping $c\colon F_B \to B$. We define a~structure of a~$\var V$-algebra on
the set $\rel B$ as a~retraction (see \cite[Definition 4.1]{barto.oprsal.ea15})
of $\al F_B$ by $i$ and $c$ where $i\colon B \to F_B$ is an inclusion ($i(b) =
x_b$, resp.), i.e.,
\[
  f^{\al B} (b_1,\dots,b_n) = c( f^{\al F_B} ( x_{b_1}, \dots, x_{b_n} ) ).
\]
One can observe that operations defined this way satisfy all linear identities
satisfied by $\al F_B$ (see also \cite[Corollary 5.4]{barto.oprsal.ea15}),
therefore $\al B \in \var V$. To prove that $\al B$ is compatible with $\rel B$,
we consider a~relation $R^{\rel B}$ of $\rel B$ and $\tup b_1,\dots,\tup b_n \in
R^{\rel B}$. Then by the definition of $\rel F_B$, $\tup b_1,\dots,\tup b_n \in
R^{\rel F_B}$, therefore $f^{\al F_B} ( \tup b_1, \dots, \tup b_n ) \in R^{\rel
F_B}$, and $c( f^{\al F_B} ( \tup b_1, \dots, \tup b_n ) ) \in R^{\rel B}$ from
the definition of coloring. This shows that $f^{\al B}$ defined as above is
a~polymorphism of $\rel B$, and hence $\al B$ is compatible with $\rel B$.

$(5) \to (3)$: Let $\al B$ be an algebra in $\var V$ compatible with $\rel B$,
and let $c\colon \al F_B \to \al B$ be the natural homomorphism from $\al F_B$
that extends the identity mapping (the mapping $x_b \mapsto b$, resp.). We claim
that this mapping is a~coloring. Indeed, if $\tup f \in R^{\rel F_B}$ then there
is a~term $f$ and $\tup b_1,\dots, \tup b_n \in R^{\rel B}$ such that $\tup f =
f^{\al F_B}( \tup b_1,\dots,\tup b_n)$, therefore
\[
  c( \tup f ) = c( f^{\al F_B}( \tup b_1,\dots,\tup b_n) ) = f^{\al B} (\tup
  b_1,\dots,\tup b_n) \in R^{\rel B},
\]
since $\rel B$ is compatible with $\al B$. The coloring $c$ is strong by
definition.
\end{proof}

\subsection{Coloring by transitive relations}

Some of the relational structures we use for coloring have binary transitive
relations, or even equivalence relations.  In these cases, we can refine the
free structure $\rel F$ generated by a~relational structure $\rel B$ in the
following way.

\begin{lemma} \label{lem:coloring-by-quosets}
  Let $\rel B$ be a~relational structure, $\var V$ a~variety, $\rel F$ the free
  structure in $\var V$ generated by $\rel B$, and let $\rel F'$ be a~structure
  obtained from $\rel F$ by replacing every relation $R^{\rel F}$, for which
  $R^{\rel B}$ is transitive, by its transitive closure. Then a~mapping $c\colon
  F \to B$ is a~coloring if and only if it is a~homomorphism $c\colon \rel F'
  \to \rel B$.
\end{lemma}

\begin{proof}
  This directly follows from the corresponding result about relational
  structures in general. In detail, given that $\rel A$ and $\rel B$ are
  relational structures sharing a~signature such that $R^{\rel B}$ is transitive,
  then every homomorphism from $\rel A$ to $\rel B$ maps transitive closure of
  $R^{\rel A}$ to $R^{\rel B}$. Conversely, if $\rel A'$ is obtained from $\rel A$
  as described in the proposition, then any homomorphism from $\rel A'$ to $\rel
  B$ is automatically a~homomorphism from $\rel A$ to $\rel B$.
\end{proof}

Following the notation of the lemma, note that, if the relation $R^{\rel B}$ is
symmetric binary relation, then also $R^{\rel F}$ is. This follows from
the fact that if $(f,g) \in R^{\rel F}$ then there is a~term $t$ of $\var V$ and
tuples $(a_1,b_1),\dots,(a_n,b_n) \in R^{\rel B}$ such that
\(
  t(a_1,\dots,a_n) = f \) and \(
  t(b_1,\dots,b_n) = g.
\)
Applying the same term $t$ to pairs $(b_1,a_1),\dots,(b_n,a_n) \in R^{\rel B}$,
we get that $(g,f) \in R^{\rel F}$. From this observation and the lemma above,
we get that if $R^{\rel B}$ is an equivalence relation, we can take $R^{\rel F}$
to be the congruence generated by $R^{\rel B}$. This case is further described
in the following lemma.

\begin{lemma} \label{lem:coloring-by-equivalences}
  Let $\rel B$ be a~finite relational structure with $B = \{1,\dots,n\}$,
  $\alpha^{\rel B}$ an~equivalence relation,
  $\al F$ the free $B$-generated algebra,
  and $\rel F$ the free structure in $\var V$ generated by $\rel B$. Then the
  the following three definitions give the same relation:
  \begin{enumerate}
    \item the transitive closure of $\alpha^{\rel F}$,
    \item the congruence of $\al F$ generated by $\alpha^{\rel B}$,
    \item the set of all pairs
      \(
        ( f^{\al F}(1,\dots,n) , g^{\al F}(1,\dots,n))
      \)
      where $f,g$ are terms of $\var V$ such that
      \(
        f(x_1,\dots,x_n) \equals g(x_1,\dots,x_n)
      \)
      is satisfied in $\var V$ whenever $x_i \equals x_j$ for all $(i,j) \in
      \alpha^{\rel B}$.
  \end{enumerate}
\end{lemma}

\begin{proof}
  Let us denote the relation defined in the item ($i$) of the above list by
  $\alpha_i$.  Clearly, $\alpha_1 \subseteq \alpha_2$ since $\alpha_2$ is
  transitive and contains $\alpha^{\rel F}$. We will prove $\alpha_2 \subseteq
  \alpha_3 \subseteq \alpha_1$.

  $\alpha_2 \subseteq \alpha_3$: Clearly, $\alpha_3$ is symmetric,
  transitive, and reflexive. It is straightforward to check that it is
  compatible with operations of $\al F$.
  Finally, we claim that $\alpha^{\rel B} \subseteq \alpha_3$. Indeed, pick
  $(i,j) \in \alpha^{\rel B}$ and choose $f(x_1,\dots,x_n) = x_i$ and
  $g(x_1,\dots,x_n) = x_j$. Clearly, $f$ and $g$ satisfies the condition in (3),
  therefore $(i,j) = (f^{\al F}(1,\dots,n),g^{\al F}(1,\dots,n)) \in \alpha_3$.
  We proved that $\alpha_3$ is a~congruence of $\al F$ containing $\alpha^{\rel
  B}$, which concludes $\alpha_2 \subseteq \alpha_3$.

  $\alpha_3 \subseteq \alpha_1$: Let $(f^{\al F}(1,\dots,n),g^{\al
  F}(1,\dots,n))$ be a~typical pair in $\alpha_3$. Fix a~representative for each
  $\alpha^{\rel B}$ class, and let $i_k$ denote the representative of the class
  containing~$k$. From the definition of $\alpha^{\rel F}$, we obtain
  \(
    ( f^{\al F}(1,\dots,n), f^{\al F}({i_1},\dots,{i_n}) )
      \in \alpha_{\rel F}
  \) and 
  \(
    ( g^{\al F}({i_1},\dots,{i_n}), g^{\al F}(1,\dots,n) )
    \in \alpha_{\rel F}
  \).
  Since $f(x_{i_1},\dots,x_{i_n}) \equals g(x_{i_1},\dots,x_{i_n})$ is true
  in $\var V$, we get that $f^{\al F}({i_1},\dots,{i_n}) = g^{\al
  F}({i_1},\dots,{i_n})$, and consequently
  \[
    \bigl( f^{\al F}(1,\dots,n), g^{\al F}(1,\dots,n) \bigr) \in
      \alpha^{\rel F} \circ \alpha^{\rel F} \subseteq \alpha_1
  \]
  which concludes the proof of $\alpha_3 \subseteq \alpha_1$.
\end{proof}

As a~consequence of the previous lemma, we also obtain that coloring by
a~relational structure $\rel B$ whose all relations are equivalence relation is
equivalent to having terms compatible with projections as defined in
\cite{sequeira06} (compare item (3) with Definition 5.2 of \cite{sequeira06}).
The corresponding set $(\alpha_1,\dots,\alpha_k)$ of projections can be obtained
from the relations of $\rel B$.

\subsection{Tarski's construction}
  \label{sec:tarski-construction}

In this section we describe a~transfinite construction that appeared in
\cite{burris71} and is attributed to Tarski. This construction is an algebraical
version of L\"owenheim-Skolem theorem, and we will use it to show that the
chains of clones $\clo P_\kappa$, $\clo B_\kappa$, and $\clo C_\kappa$ defined in
the following sections are indeed chains as ordered by an existence of a~clone
homomorphism.

Given an~algebra $\al A$, we define a~sequence $\al A_\lambda$ indexed by
ordinals by the following transfinite construction:
We start with $\al A_0 = \al A$. For an ordinal successor $\lambda + 1$ define
$\al A_{\lambda+1}$ as an algebra isomorphic to $\al A_\lambda^2$ while
identifying the diagonal with $\al A_\lambda$, i.e., take $A_{\lambda+1}
\subseteq A_\lambda$ with a~bijection $f_\lambda\colon A_\lambda^2 \to
A_{\lambda+1}$ such that $f_\lambda(a,a) = a$, and define the structure of $\al
A_{\lambda+1}$ in such a way that $f_\lambda$ is an~isomorphism.
For a~limit ordinal $\lambda$, we set $\al A_\lambda =
\Union_{\alpha < \lambda} \al A_\alpha$.
This construction produces algebras in the variety generated by $\al A$ of all
infinite cardinalities larger than $|A|$. Indeed, one can observe that
$|A_\lambda| = \max\{ |\lambda|, |A| \}$ for all infinite $\lambda$.
We can also follow this construction to get from a~fixed proper subset
$U\subset A$, $U\neq \emptyset$ a~set $U_\lambda \subseteq A_\lambda$ such that both its
cardinality and cardinality of its complement is $|\lambda|$ (given that
$|\lambda| \geq |A|$):
We put $U_0 = U$;
for ordinal successor $\lambda+1$, we take $U_{\lambda + 1} = f_\lambda( A_\lambda
\times U_\lambda )$ (observe that $U_{\lambda+1} \cap A_\lambda = U_\lambda$);
and for limit $\lambda$, we take $U_\lambda = \Union_{\alpha < \lambda}
U_\alpha$.
 
\section{Taylor's conjecture}

In this section, we prove Theorem~\ref{thm:taylors-conjecture-introduction}
about Taylor's conjecture on congruence modular varieties. 
Congruence modular varieties have been thoroughly investigated, and they have
many nice properties, nevertheless we need only the definition and the following
Mal'cev characterization of these varieties by A.~Day \cite{day69}:

\begin{theorem} \label{thm:day}
The following are equivalent for any variety $\var V$.
\begin{enumerate}
  \item Every algebra in $\var V$ has modular congruence lattice;
  \item there exists $n$, and quaternary terms $d_0,\dots,d_n$ 
    such that the following identities are satisfied in $\var V$:
    \begin{align*}
      d_0(x,y,z,w) &\equals x \text{ and } d_n(x,y,z,w) \equals w, \\
      d_i(x,y,y,z) &\equals d_{i+1}(x,y,y,z) \text{ for even $i$,} \\
      d_i(x,x,y,y) &\equals d_{i+1}(x,x,y,y) \text{ and }
      d_i(x,y,y,x) \equals d_{i+1}(x,y,y,x) \text{ for odd $i$}.
    \end{align*}
\end{enumerate}
\end{theorem}

A~sequence of terms satisfying the identities in (2) is referred to as
\emph{Day terms}; we will also say that some functions (or polymorphisms) are
\emph{Day functions} if they satisfy these identities.

\subsection{Pentagons}

For the description of a~cofinal chain of clones in the complement of the
filter of clones containing Day terms, we will use relational structures
of the form $\rel P = (P; \alpha, \beta, \gamma)$ where $\alpha$, $\beta$, and
$\gamma$ are equivalence relations on $P$ that do not satisfy the modularity
law.
A~very similar structures have been used in \cite{bova.chen.valeriote13} to
prove that the problem of comparison of pp-formulae is coNP-hard for algebras
that do not generate congruence modular varieties; it was also used in
\cite{mcgarry09}. The following definition of a~pentagon is almost identical to
the one in \cite{bova.chen.valeriote13} (we focus on those pentagon which are by
themselves `interesting'). For our purpose we need pentagons of even more
special shape; we call them special and very special pentagons.

\begin{definition} A~\emph{pentagon} is a~relational structure
$\rel P$ in the signature $\{\alpha,\beta,\gamma\}$ with three binary relations
that are all equivalence relations on $P$ satisfying:
\begin{itemize}
  \item $\alpha^{\rel P} \meet \beta^{\rel P} = 0_P$,
  \item $\alpha^{\rel P} \circ \beta^{\rel P} = 1_P$,
  \item $\gamma^{\rel P} \join \beta^{\rel P} = 1_P$, and
  \item $\gamma^{\rel P} < \alpha^{\rel P}$.
\end{itemize}
\end{definition}

The first two items in this definition ensure that every pentagon naturally
factors as a~direct product $P = A \times B$ in such a~way that $\alpha$ and
$\beta$ are kernels of projections on the first and the second coordinate,
respectively. In this setting, $\gamma$ is an equivalence relation that relates
some pairs of the form $((a,b),(a,c))$. For an~equivalence $\gamma$ on a~product
$A\times B$, and $a \in A$, we define $\gamma^a$ to be the following equivalence
on $B$:
\[
  \gamma^a := \{ (b,c) : ((a,b),(a,c)) \in \gamma \}.
\]
A~pentagon on the set $A \times B$ (with $\alpha$ and $\beta$ being the two
kernels of projections) is said to be \emph{special} if $\gamma^a$ for $a\in A$
gives exactly two distinct congruences with one of them being the full
congruence on $B$, i.e., there exists $\eta < 1_B$ such that
\[
  \{ \gamma^a : a\in A \} = \{ 1_B, \eta \}.
\]
Such~pentagon is \emph{very special} if the above is true for $\eta = 0_B$. When
defining a~special, or a~very special pentagon $\rel P$ on a~product $A\times
B$, we will usually speak only about the relation $\gamma^{\rel P}$ since the
other two relations are implicitly given as the kernels of projections.

In this section we will use pentagons $\rel P_0$ and $\rel P_\kappa$ where
$\kappa$ is an infinite cardinal. All of these pentagons are either themselves
very special, or isomorphic to a~very special pentagon.

\begin{definition} \label{def:pentagon_0}
We define $\rel P_0$ as the~smallest possible pentagon. In detail: $P_0 =
\{0,1,2,3\}$ and $\alpha^{\rel P_0}$, $\beta^{\rel P_0}$, and $\gamma^{\rel
P_0}$ are equivalences defined by partitions $12|03$, $01|23$, and $12|0|3$,
respectively.

\begin{figure}[ht]
  \begin{tikzpicture}
    \node (0) at (0,0) {0};
    \node (1) at (2.5,0) {1};
    \node (2) at (2.5,2.5) {2};
    \node (3) at (0,2.5) {3};

    \draw (0) edge [bend left] node[midway,left] {$\alpha$} (3);
    \draw (1) edge [bend left] node[midway,left] {$\alpha$} (2);

    \draw (1) edge [bend right] node[midway,right] {$\gamma$} (2);

    \draw (0) edge node[midway,above] {$\beta$} (1);
    \draw (2) edge node[midway,above] {$\beta$} (3);
  \end{tikzpicture}
  \caption{The pentagon $\rel P_0$}
\end{figure}
\end{definition}

The pentagon $\rel P_0$ is isomorphic to a~very special pentagon on the set $P =
\{0,1\} \times \{0,1\}$ with $\gamma^0 = 0_{\{0,1\}}$ and $\gamma^1 =
1_{\{0,1\}}$; an isomorphism is given by the map: $0 \mapsto (0,0)$, $1 \mapsto
(1,0)$, $2 \mapsto (1,1)$, and $3 \mapsto (0,1)$.

\begin{definition} \label{def:pentagon_kappa}
For an~infinite cardinal $\kappa$, we define a~pentagon $\rel P_\kappa$: Fix
$U_\kappa \subseteq \kappa$ with $|U_\kappa| = |\kappa \setminus U_\kappa| =
\kappa$, and define $P_\kappa = \kappa \times \kappa$, $\alpha^{\rel P_\kappa}$
and $\beta^{\rel P_\kappa}$ to be the kernels of the first and the second
projection, respectively, and
\[
  \gamma^{\rel P_\kappa} = \{ ((a,b),(a,c)) : a,b,c\in\kappa
    \text{ such that } a \in U_\kappa \text{ or } b = c \}
.\]
\end{definition}

\subsection{Coloring and linear varieties}

Using Day terms, we obtain a~characterization of congruence non-modular
varieties by the means of coloring by the pentagon $\rel P_0$ defined above.

\begin{proposition} \label{prop:coloring-for-modularity}
A~variety $\var V$ does not have Day terms if and only if it has strongly $\rel
P_0$-colorable terms.
\end{proposition}

\begin{proof}
  Let $\rel F$ denote the structure obtained from the free structure in $\var
  V$ generated by $\rel P_0$ by replacing every its relation by its transitive
  closure, and note that all three relations of $\rel F$ are equivalences (see
  Lemma~\ref{lem:coloring-by-equivalences}), and also that $c\colon F \to P_0$ is
  a~coloring if and only if it is a~homomorphism from $\rel F$ to $\rel P_0$
  (see Lemma~\ref{lem:coloring-by-quosets}). To simplify the notation we
  identify a~4-ary term $d$ with the element \hbox{$d(x_0,x_1,x_2,x_3) \in F$}.

  First, we prove the implication from left to right. For a~contradiction,
  suppose that $\var V$ has Day terms $d_0,\dots,d_n$, and that there is
  a~strong coloring $c\colon \rel F \to \rel P_0$.  Any strong coloring $c$ maps
  $x_0$, and therefore also $d_0$, to $0$, and it satisfies \( (
  c(d_i),c(d_{i+1}) ) \in \gamma \) for even $i$, and \( ( c(d_i),c(d_{i+1}) )
  \in \alpha \meet \beta \) for odd $i$; this follows from Day identities that
  can be reformulated as: $(d_i,d_{i+1}) \in \gamma^{\rel F}$ for even $i$ and
  $(d_i,d_{i+1}) \in \alpha^{\rel F} \meet \beta^{\rel F}$ for odd $i$.  Using
  these observations, we obtain by induction on $i$ that $c$ maps $d_i$ to $0$
  for all $i$, in particular $c(x_3) = c(d_n) = 0$. This gives us
  a~contradiction with $c(x_3) = 3$.

  For the converse, we want to prove that Day terms are the only obstruction for
  having strongly $\rel P_0$-colorable terms. We will do that by defining
  a~valid strong coloring of terms of any $\var V$ that does not have Day terms.
  Observe that even in this case we can repeat the argumentation from the above
  paragraph to get that $c(f) = 0$ for any $f$ that is connected to $x_0$ by
  a~Day-like chain, i.e., terms $d_0$, \dots, $d_n$ satisfying Day identities
  where $d_n \equals x_3$ is replaced by $d_n \equals f$. Therefore, we put
  $c(f) = 0$ if this is the case.
  Furthermore, for any $f$ with $(f,g) \in \beta^{\rel F}$ for some $g$ having
  the property above, we have $c(f) \in \{0,1\}$. We put $c(f) = 1$ if this is
  the case and we have not defined the value $c(f)$ yet. Similarly, if $(f,g)
  \in \alpha^{\rel F}$ for some $g$ with $c(g) = 0$ and $c(g)$ undefined, we put
  $c(f) = 3$ following the rule that $c(f) \in \{0,3\}$. Note that in this step
  we have defined $c(x_3) = 3$ since $(x_0,x_3) \in \alpha^{\rel F}$. Finally,
  for the remaining $f \in F$ we put $c(f) = 2$. This definition is summarized
  as follows:
  $c(f) = 0$ if $(f,x_0) \in (\freeclo\alpha \meet \freeclo\beta) \join
    \freeclo \gamma $, otherwise:
  \[
    c(f) = \begin{cases}
      1 & \text{if $(f,g) \in \freeclo\beta$ for some $g$ with $c(g) = 0$,}\\
      3 & \text{if $(f,g) \in \freeclo\alpha$ for some $g$ with $c(g) = 0$, 
          and} \\
      2 & \text{in all remaining cases.}
    \end{cases}
  \]
  Note that if $f$ satisfies both the first and the second row of the
  above, say there are $g_1$, $g_2$ with $c(g_i) = 0$, $(f,g_1) \in
  \beta^{\rel F}$, and $(f,g_2) \in \alpha^{\rel F}$, then also $(f,g_1) \in
  \alpha^{\rel F}$ since
  \(
    (\freeclo\alpha \meet \freeclo\beta) \join \freeclo \gamma
    \leq \freeclo \alpha
  \), and $c(f)$ is assigned value $0$ in the first step. Therefore, $c$ is
  well-defined.

  We claim that $c$ defined this way is a~homomorphism from $\rel F$ to
  $\rel P_0$, and therefore a~coloring. To prove that we need to show that $c$
  preserves all $\alpha$, $\beta$, and $\gamma$. First, let $(f,g) \in
  \alpha^{\rel F}$.
  Then either there is $h\in F$ with $c(h) = 0$ in the
  $\alpha^{\rel F}$-class of $f$ and $g$, and both $c(f)$ and $c(g)$ are
  assigned values in $\{0,3\}$, or there is no such $h$ and both $c(f)$ and
  $c(g)$ are assigned values in $\{1,2\}$. Either way $(c(f),c(g)) \in
  \alpha^{\rel P_0}$. The same argument can be used for showing $(c(f),c(g)) \in
  \beta^{\rel P_0}$ given that $(f,g) \in \beta^{\rel F}$. Finally, suppose that
  $(f,g) \in \gamma^{\rel F}$. If $c(f) \in \{1,2\}$, we immediately get that
  $c(g) \in \{1,2\}$ since $c$ preserves $\alpha$. If this is not the case, and
  $f$ and $g$ are assigned values that are not in the same class of $\gamma^{\rel
  P_0}$, then $c(f) = 0$ and $c(g) = 3$, or vice-versa. But this implies that
  $(f,g) \notin \gamma^{\rel F}$ as otherwise both would be assigned the value
  $0$. This concludes that $c$ is a~coloring.

  We are left to prove that $c$ is strong, i.e., $c(x_i) = i$ for $i=0,1,2,3$. 
  We know (from the defintion of $c$) that $c(x_0) = 0$ and $c(x_3) = 3$.
  Now, $x_1$ and $x_2$ are assigned values that are in the
  same class of $\gamma^{\rel P_0}$, and moreover $(c(x_1),c(x_0))\in
  \beta^{\rel P_0}$ and $(c(x_3),c(x_2))\in \beta^{\rel P_0}$. This leaves us
  with the only option: $c(x_1) = 1$ and $c(x_2) = 2$.
\end{proof}

Note that in the previous proof we use the same three congruences ($\alpha^{\rel
F}$, $\beta^{\rel F}$, and $\gamma^{\rel F}$) of the free algebra generated by
a~four element set that also appeared in the original proof of Day's results.
This is not a pure coincidence, we will use a~very similar method of deriving
a~suitable structure for coloring in Section~\ref{sec:n-permutability}.

\subsection{Idempotent varieties}
  \label{sec:modularity.idempotent}

In this subsection, we prove a~generalization of the following theorem of McGarry
\cite{mcgarry09} for idempotent varieties that do not need to be locally finite. 

\begin{theorem}[\cite{mcgarry09}]
A~locally finite idempotent variety is not congruence modular if and
only if it contains and algebra compatible with a~special pentagon.
\qed
\end{theorem}

Throughout this proof we will work with several variants of special and
very special pentagons $\rel P$ defined on a~product $A \times B$ for two algebras $\al
A$, $\al B$ in our idempotent variety, and in particular with the corresponding
relations $\gamma^{\rel P}$.
We say that $\gamma$ is a~\emph{modularity blocker} if $\rel P =
(A\times B; \alpha, \beta, \gamma)$ is a~special pentagon given that $\alpha$
and $\beta$ are kernels of the two projections, and $\gamma$ is a~\emph{special
modularity blocker} if $\rel P$ is a~very special pentagon.

The first step of the proof coincides with McGarry's proof. Nevertheless, we
present an alternative proof for completeness.

\begin{lemma}
Let $\mathcal V$ be an~idempotent variety, and $\al F$ the free algebra in $\var
V$ generated by $\{x,y\}$. Then $\mathcal V$ is congruence modular if and only if
\[
  \big((y,x),(y,y)\big) \in \Cg_{\al F\times \al F} \{ \big((x,x),(x,y)\big) \}.
    \eqno{(*)}
\]
\end{lemma}

\begin{proof}
Given that $\var V$ is congruence modular, one obtains $(*)$ from the modularity
law applied on the two kernels of projections and the congruence on the right
hand side of $(*)$.

To show that $(*)$ also implies congruence modularity, we will prove that $(*)$
is not true in any $\mathcal V$ that is idempotent and not congruence modular.
First, we consider the free algebra $\al F_4$ in $\var V$ generated by the
four-element set $\{x_0,x_1,x_2,x_3\}$ and its congruences $\alpha = \Cg
\{(x_0,x_3),(x_1,x_2)\}$, $\beta = \Cg\{(x_0,x_1),(x_2,x_3)\}$, and $\gamma = \Cg\{ (x_1,x_2) \}$.
Note that this is the structure $\rel F$ that has been used in the proof of
Proposition~\ref{prop:coloring-for-modularity}. Either from this proposition, or
by a~standard argument from the proof of Day's result, we obtain that $\alpha$,
$\beta$, and $\gamma$ don't satisfy the modularity law, and in particular,
$(x_0,x_3) \notin \gamma \join (\alpha \meet \beta)$.

Next, we shift this property to the second power of the two-generated free
algebra; consider the homomorphism $h\colon \al F_4 \to \al F \times \al F$
defined on the generators by
$x_0 \mapsto (y,x)$, $x_1 \mapsto (x,x)$, $x_2 \mapsto (x,y)$, and $x_3 \mapsto (y,y)$.
The homomorphism $h$ is surjective, since for every two binary idempotent terms
$t,s$ we have
\(
  h\big( t(s(x_1,x_2),s(x_0,x_3)) \big) = \big( t(x,y), s(x,y) \big)
\).
Finally, since the kernel of $h$ is $\alpha \meet \beta$, we get that 
\[
  h^{-1}\big(\Cg_{\al F\times \al F} \{((x,x),(x,y))\}\big)
    = \gamma \join (\alpha \meet \beta) \notni (x_0,x_3),
\]
and consequently $((y,x),(y,y)) = (h(x_0),h(x_3)) \notin \Cg_{\al F\times \al F}
\{((x,x),(x,y))\}$.
\end{proof}

\begin{lemma} \label{lem:idempotent-modularity-blocker}
Let $\mathcal V$ be an~idempotent variety which is not congruence modular, and
$\al F = \al F_{\mathcal V} (x,y)$. Then there is a~modularity blocker $\gamma$
in $\al F\times \al F$ such that $((x,x),(x,y)) \in \gamma$ and $((y,x),(y,y))
\notin \gamma$,
\end{lemma}

\begin{proof}
Let $\gamma_0 = \Cg_{\al F\times \al F} \{ \big((x,x),(x,y)\big) \}$. From the previous
lemma, we know that $((y,x),(y,y)) \notin \gamma_0$. Let $\eta$ be a~maximal
equivalence relation on $F$ such that
\[
  \big((y,x),(y,y)\big) \notin
    \gamma_0 \join \Cg \{ \big((y,a),(y,b)\big) : (a,b) \in \eta \}
\]
(such equivalence exists from Zorn's lemma). We claim that the equivalence on
the right-hand side, let us call it $\gamma$, is a~modularity blocker. This is
proven in two steps:

{\def\qedsymbol{$\triangle$}
\begin{claim} $\gamma^p \geq \eta$ for all $p\in F$.
\end{claim}

Let $f$ be binary term such that $f(x,y) = p$.
Observe that $\gamma^x = 1_F \geq \eta$ and $\gamma^y = \eta$, therefore we
obtain that
\(
  \big((p,a),(p,b)\big) =
    \big(f\big((x,a),(y,a)\big),f\big((x,b),(y,b)\big)\big) \in \eta
\)
for all
$(a,b) \in \eta$. Which shows that $\gamma^p \geq \eta$. \qed

\begin{claim} If $\gamma^p > \eta$ for some $p\in F$ then $\gamma^p = 1_F$.
\end{claim}

Let $e = e' \times 1_{\al F}$, where $e'\colon \al F \to \al F$ is the
homomorphism defined by $x\mapsto x$ and $y\mapsto p$, and consider the
congruence $\gamma_1 = e^{-1}(\gamma)$. From this definition, we obtain that
$\gamma_1^y \geq \gamma^p > \eta$ and moreover $\gamma_1 \geq \gamma_0$ since
$(e(x,x),e(x,y)) = ((x,x),(x,y)) \in \gamma$.  Therefore, from the maximality
of $\eta$, we obtain that $((y,x),(y,y)) \in \gamma_1$, and as a~consequence
thereof, $((p,x),(p,y)) = (e(y,x),e(y,y)) \in \gamma$. This shows that
$\gamma^p = 1_\al F$ and completes the proof of the second claim. \qed
}

By combining both claims, we get that $\gamma^p \in \{ \eta, 1_F \}$, therefore
$\gamma$ is a~modularity blocker.
\end{proof}
The above proof was inspired by the proof of Lemma 2.8
from~\cite{kearnes.tschantz07}.

\begin{corollary} \label{cor:idempotent-special-pentagon}
  Every idempotent variety that is not congruence modular is
  compatible with a~very special pentagon.
\end{corollary}

\begin{proof}
Following the notation of the previous lemma, we know that $(F\times F,\alpha,
\beta,\gamma)$ is a~special pentagon for kernels of projections $\alpha$,
$\beta$.
To obtain a~very special pentagon $\rel P$, first observe that $\gamma^y$ is
a~congruence on $\al F$: by idempotence we get that $\{y\} \times \al F$ is
a~subalgebra of $\al F^2$ isomorphic to $\al F$, and $\gamma^y$ is the image of
the restriction of $\gamma$ to $\{y\} \times \al F$ under this isomorphism.
We set $\al A = \al F$, $\al B = \al F / \gamma^y$, and $P = A \times B$.
Finally, the relation $\gamma^{\rel P}$ is defined as the image of $\gamma$
under the natural epimorphism from $\al F^2$ to $\al A\times \al B$.
\end{proof}

The next step is to show that we can increase the size of compatible very
special pentagons. In other words, that every variety compatible with a~very
special pentagon is compatible with $\rel P_\kappa$'s for all big enough
cardinals $\kappa$.

\begin{proposition} \label{prop:huge-pentagons}
If a~variety $\var V$ contains an algebra compatible with a~very special
pentagon $\rel P$, then it contains an algebra compatible with $\rel P_\kappa$
for all $\kappa \geq \aleph_0 + |P|$.
\end{proposition}

\begin{proof}
  Let $\al A$ and $\al B$ be two algebras in $\var V$ such that $\rel P$ is
  compatible with $\al A \times \al B$, let $U\subseteq A$ denote the set
  $\{ a\in A : \gamma^a = 1_B \}$, and let $\kappa \geq \aleph_0 + |P|$ be
  a~cardinal.

  We will construct a~pentagon which is isomorphic to $\rel P_\kappa$ by
  a~variation on Tarski's construction described in
  Section~\ref{sec:tarski-construction}.  Let $\al A_\lambda$ and $U_\lambda$ be
  the algebra and its subset obtained by the construction from $\al A$ and $U$
  in $\lambda$ steps, and let $\al B_\lambda$ be the algebra obtained by the
  construction from $\al B$ in $\lambda$ steps. We define an~equivalence
  relation $\gamma_\lambda$ on $A_\lambda \times B_\lambda$ by
  \[
    \gamma_\lambda = \{ ((a,b), (a',b')) :
      \text{$a = a'$ and if $a\not\in U_\lambda$ then $b = b'$} \}.
  \]
  We will prove that $\gamma_\lambda$ is a~congruence of $\al A_\lambda\times
  \al B_\lambda$ by a~transfinite induction: $\gamma_0 = \gamma$ is
  a~congruence. For the induction step, suppose that $\gamma_\lambda$ is
  a~congruence. The algebras $\al A_{\lambda + 1}$ and $\al B_{\lambda +1}$ are
  defined as being isomorphic to the second powers of $\al A_\lambda$ and $\al
  B_\lambda$ respectively. Let us suppose that they are in fact the second
  powers themselves and $U_{\lambda+1} = A_\lambda \times U_\lambda$. Then
  $\gamma_{\lambda+1}$ can be described as relating those pairs of pairs
  $\big((a_1,a_2),(b_1,b_2)\big)$ and $\big((a'_1,a'_2),(b'_1,b'_2)\big)$ that satisfy:
  $a_1 = a'_1$, $a_2 = a'_2$, and if $a_2\not\in U_\lambda$ then $b_1 =
  b'_1$ and $b_2 = b'_2$. Rewriting this condition, we get
  \begin{multline*}
    \big(\big((a_1,a_2),(b_1,b_2)\big),
    \big((a'_1,a'_2),(b'_1,b'_2)\big)\big)
    \in \gamma_{\lambda+1}
    \text{ if and only if }
    \\
    a_1 = a'_1  \text{, }
    a_2 = a'_2  \text{, }
    \big((a_2,b_1),(a'_2,b'_1)\big) \in \gamma_\lambda \text{, and }
    \big((a_2,b_2),(a'_2,b'_2)\big) \in \gamma_\lambda
  \end{multline*}
  which is clearly compatible with the operations, since $\gamma_\lambda$ is.
  Also note, that if we restrict $\gamma_{\lambda+1}$ to $\al A_{\lambda} \times
  \al B_{\lambda}$, we get $\gamma_\lambda$. For limit ordinals $\lambda$, the
  compatibility of $\gamma_\lambda$ is obtained by a~standard compactness
  argument.
  This way, we obtain for every $\lambda$ a~very special pentagon $\rel
  P'_\lambda$ with $P'_\lambda = A_\lambda \times B_\lambda$ and $\gamma^{\rel
  P_\lambda} = \gamma_\lambda$.

  The final step is to prove that the pentagon $\rel P'_\lambda$ is isomorphic
  to $\rel P_\kappa$ for $\kappa = |\lambda|$. Indeed, $P_\kappa = \kappa \times
  \kappa$, $P'_\lambda = A_\lambda \times B_\lambda$ with $|A_\lambda| =
  |B_\lambda| = \kappa$. Also compare
  \[
    \gamma^{\rel P_\kappa} = \{ ((a,b),(a',b')) :
      \text{ $a = a'$ and if $a\not\in U_\kappa$ then $b = b'$ } \}
  \]
  with the above definition of $\gamma_\lambda$. From these observations, it is
  immediate that any bijection $a\times b\colon \kappa \times \kappa \to
  A_\lambda \times B_\lambda$ defined by components $a$ and $b$ such that $a$
  maps $U_\kappa$ onto $U_\lambda$ is an isomorphism between $\rel P_\kappa$ and
  $\rel P'_\lambda$.
\end{proof}

\subsection{Proof of Theorem~\ref{thm:overview-introduction}(i)}
  Given a~variety $\var V$ that is not congruence modular, we distinguish
  between two cases: $\var V$ is linear; and $\var V$ is idempotent. If it is
  the first case, we get from Proposition~\ref{prop:coloring-for-modularity} and
  Lemma~\ref{lem:linear-varieties} that $\var V$ contains an algebra compatible
  with a~very special pentagon $\rel P^{\var V} = \rel P_0$. In the second case,
  $\var V$ is idempotent, we obtain a~very special pentagon $\rel P^{\var V}$
  compatible with an~algebra from $\var V$ by
  Corollary~\ref{cor:idempotent-special-pentagon}.

  Given that both $\var V$ and $\var W$ are not congruence modular, and either
  linear, or idempotent, we have very special pentagons $\rel P^{\var V}$ and
  $\rel P^{\var W}$ compatible with algebras from the corresponding varieties.
  Choosing an infinite cardinal $\kappa$ larger than the sizes of both $P^{\var
  V}$ and $P^{\var W}$.  Proposition~\ref{prop:huge-pentagons} yields that both
  varieties contain algebras compatible with the pentagon $\rel P_\kappa$,
  therefore by Lemma~\ref{lem:joins}, the interpretability join $\var V \join
  \var W$ contains an algebra compatible with $\rel P_\kappa$ witnessing that
  $\var V \join \var W$ is not congruence modular. \qed

\subsection{Cofinal chain}

Here we investigate properties of the transfinite sequence
\(\clo P_0\), \(\clo P_{\aleph_0}\), \(\clo P_{\aleph_1}, \dots\)
of polymorphism clones of pentagons $\rel P_0, \rel P_{\aleph_0}, \dots$;
in particular, we show that this sequence form a~strictly increasing chain in
the lattice of clones, and as a~corollary thereof, we obtain that there is no
maximal (in the interpretability order) idempotent variety that is not
congruence modular. For the rest of this section, $\kappa$ and $\lambda$ will
denote either infinite cardinals, or $0$.

The fact that the chain is increasing follows from the following corollary of
Proposition~\ref{prop:huge-pentagons}.

\begin{corollary} \label{cor:modularity-chain}
There is a~clone homomorphisms from $\clo P_\lambda$ to $\clo P_\kappa$ for all
$\lambda \leq \kappa$.
\end{corollary}

\begin{proof}
  If we use Proposition~\ref{prop:huge-pentagons} on the variety of actions of
  $\clo P_\lambda$, we obtain that it contains an algebra compatible with $\rel
  P_\kappa$. But such an algebra corresponds to an~action of $\clo P_\lambda$ on
  $\rel P_\kappa$ by polymorphisms, therefore to a~clone homomorphism from $\clo
  P_\lambda$ to $\clo P_\kappa$.
\end{proof}

Before we proceed further, we need a~better description of polymorphisms of the
pentagons $\rel P_\kappa$.

\begin{lemma} \label{lem:compatibility-with-blocker}
Let $\kappa$ be an~infinite cardinal, and consider the pentagon $\rel P_\kappa$,
and let $f\colon P_\kappa^k \to P_\kappa$ be a~functions defined by components
$f^1,f^2\colon \kappa^k \to \kappa$. Then the following are equivalent:
\begin{enumerate}
  \item $f$ is a polymorphism of $\rel P_\kappa$,
  \item if $f^1 (a_1,\dots,a_n) \notin U$ for some $a_1$, \dots, $a_n \in
  \kappa$ then $f^2$ does not depend on any coordinate $i$ such that $a_i \in
  U$.
\end{enumerate}
\end{lemma}

\begin{proof}
  $(1)\to(2)$\quad Since $f$ is defined component-wise, if is compatible with
  $\alpha^{\rel P_\kappa}$ and $\beta^{\rel P_\kappa}$. For compatibility with
  $\gamma^{\rel P_\kappa}$, recall that $((a,b),(a',b')) \in \gamma^{\rel
  P_\kappa}$ if and only if $a = a'$, and $a\in U$ or $b = b'$.  Therefore, if
  $f$ is compatible with $\gamma^{\rel P_\kappa}$ and $f^1 (a_1,\dots,a_n)
  \notin U$ then
  \(
    f^2 (b_1,\dots,b_n) =
    f^2 (b'_1,\dots,b'_n)
  \)
  for all $b_i$, $b_i'$ where $b_i = b_i'$ whenever $a_i \notin U$ (observe that
  $((a_i,b_i),(a_i,b_i'))\in \gamma^{\rel P_\kappa}$ for all $i$).  This means
  that $f^2$ does not depend on any coordinate $i$ with $a_i \in U$.
  The implication $(2) \to (1)$ is given by reversing this argument.
\end{proof}

Instead of proving that the chain of $\clo P_\kappa$'s is strictly increasing,
we prove the following stronger statement which will allow us to show that there
is no maximal (in the ordering by interpretability) idempotent non-modular
variety.

\begin{proposition} \label{prop:modularity-large-algebras}
  The idempotent reducts of $\clo P_\kappa$'s form a~strictly increasing chain
  in the lattice of clones.
\end{proposition}

\begin{proof}
Let $\idclo P_\kappa$ denote the idempotent reduct of $\clo P_\kappa$. We need
to show two facts: (1) there is a~clone homomorphism from $\idclo P_\lambda$ to
$\idclo P_\kappa$ for all $\lambda \leq \kappa$, and (2) there is no clone
homomorphism from $\idclo P_\kappa$ to $\idclo P_\lambda$ for any $\lambda <
\kappa$. (1) follows directly from Corollary \ref{cor:modularity-chain}, since any clone
homomorphism from $\clo P_\lambda$ to $\clo P_\kappa$ has to preserve idempotent
functions.

To prove (2) suppose that $\kappa > \lambda$. We will find identities that are
satisfied in $\idclo P_\kappa$ but are not satisfiable in $\idclo P_\lambda$. In
fact they are not satisfiable in any clone on a~set of cardinality strictly
smaller than $\kappa$ except the~one-element set.
The identities use binary symbols $f_i$ for $i\in \kappa$ and ternary symbols
$p_{i,j}$, $q_{i,j}$, $r_{i,j}$ for $i,j \in \kappa$, $i\neq j$:
\begin{align*}
  x &\equals p_{i,j}(x,f_j(x,y),y), \\
  p_{i,j}(x,f_i(x,y),y) &\equals q_{i,j}(x,f_j(x,y),y), \\
  q_{i,j}(x,f_i(x,y),y) &\equals r_{i,j}(x,f_j(x,y),y), \\
  r_{i,j}(x,f_i(x,y),y) &\equals y
\end{align*}
for all $i\neq j$, and $f_i(x,x) \equals x$ for all $i$.

We will define functions in $\idclo P_\kappa$ that satisfy these identities
coordinatewise (recall $P_\kappa = \kappa \times \kappa$); this will assure that
these functions are compatible with $\alpha^{\rel P_\kappa}$ and $\beta^{\rel
P_\kappa}$. We fix $c^1 \in U_\kappa$ and $c^2 \in \kappa$. The functions $f_i$
are defined, unless required by idempotence otherwise, as constants while
choosing different constants for different $i$'s. In detail, we pick $c^1_i \in
U_\kappa$ to be pairwise distinct, and $c^2_i \in \kappa$ as well. Then
\(
  f_i^u(x,y) = c^u_i
\) if $x\neq y$ and $f_i^u(x,x) = x$ for $u=1,2$.
The components of $p_{i,j}$, $q_{i,j}$, and $r_{i,j}$ are defined as
follows:
\[
\arraycolsep=0pt
\begin{array}{rlrl}
    p^1_{i,j}(x,y,z) &{}= \begin{cases}
      x & \text{if $y = f_j^1(x,z)$,}\\
      c_1 & \text{otherwise;}
    \end{cases}
  &\quad
    p^2_{i,j}(x,y,z) &{}= x;
  \\
    q^1_{i,j} (x,y,z) &{}= \begin{cases}
      x & \text{if $x = y = z$,} \\
      c_1 & \text{otherwise;}
    \end{cases}
  &\quad
    q^2_{i,j} (x,y,z) &{}= \begin{cases}
      x & \text{if $y = f_j^2 (x,z)$,} \\
      z & \text{if $y = f_i^2 (x,z)$,} \\
      c_2 & \text{otherwise;}
    \end{cases}
  \\
    r^1_{i,j}(x,y,z) &{}= \begin{cases}
      z & \text{if $y = f_i^1(x,z)$,}\\
      c_1 & \text{otherwise;}
    \end{cases}
  &\quad
    r^2_{i,j}(x,y,z) &{}= z;
\end{array}
\]
It is straightforward to check that these functions satisfy the identities. The
compatibility with $\alpha^{\rel P_\kappa}$ and $\beta^{\rel P_\kappa}$ is
immediate from defining them component-wise. The compatibility with $\gamma^{\rel
P_\kappa}$ follows from Lemma~\ref{lem:compatibility-with-blocker}: the only
case when $p^1_{i,j}(x,y,z) \notin U_\kappa$ is when $x \notin U_\kappa$ and $y
= f_j^1(x,z)$ but $p^2_{i,j}$ does not depend on the second and third variable,
therefore $p_{i,j}$ is compatible; the argument for $r_{i,j}$ is analogous;
$q_{i,j}^1(x,y,z)$ falls in $U_\kappa$ unless $x = y = z \notin U_\kappa$, so
$q_{i,j}$ is also compatible.

Next, we claim that the above identities are not satisfied in any non-trivial
algebra $\al A$ of size strictly less than $\kappa$.
Indeed, if $A$ contains two distinct elements $a$ and $b$, then $f_i(a,b)$,
$i\in \kappa$ are pairwise distinct, since if $f_i^{\al A}(a,b) = f_j^{\al
A}(a,b)$, then
\begin{multline*}
  a
  = p^{\al A}_{i,j}(a,f_j(a,b),b)
  = p^{\al A}_{i,j}(a,f_i(a,b),b)
  = q^{\al A}_{i,j}(a,f_j(a,b),b)
  = \\ q^{\al A}_{i,j}(a,f_i(a,b),b)
  = r^{\al A}_{i,j}(a,f_j(a,b),b)
  = r^{\al A}_{i,j}(a,f_i(a,b),b)
  = b.
\end{multline*}
This shows in particular that these identities are not satisfiable in $\idclo
P_\lambda$.
\end{proof}

\begin{corollary} \label{cor:no-strongest-non-modular-variety}
The class of all interpretability classes of idempotent varieties that are not
congruence modular does not have a~largest element.
\end{corollary}

\begin{proof}
  For a~contradiction, suppose that there is a~largest interpretability class
  among those containing a~non-modular idempotent variety, and let $\var V$ be
  a~variety from this class. By Corollary~\ref{cor:idempotent-special-pentagon}
  and Proposition~\ref{prop:huge-pentagons}, we know that $\var V$ is
  interpretable in the variety of actions of $\clo P_\kappa$ for some $\kappa$.
  Fix any such $\kappa$, and let $\lambda > \kappa$ be a~cardinal. By the
  maximality of $\var V$, we have that the variety of actions of $\idclo
  P_\lambda$ (the idempotent reduct of $\clo P_\lambda$) is interpretable in
  $\var V$, therefore also in the variety of actions of $\clo P_\kappa$ which
  contradicts the previous proposition.
\end{proof}
 \section{Having \texorpdfstring{$n$}{n}-permutable congruences}
  \label{sec:n-permutability}

In this section, we will investigate both the strong Mal'cev conditions for
being congruence $n$-permutable (every two congruences $\alpha$, $\beta$ of
a~single algebra satisfy \( \alpha \circ_n \beta = \beta \circ_n \alpha \)) for
a~fixed $n$, and a~general condition of being congruence $n$-permutable for some
$n$. Therefore, we will speak about a~countable chain of filters in the
interpretability lattice and its limit. The primeness of the limit filter would
be implied by primeness of each of the filters from the chain.  Unfortunately,
we are not able to establish the version of Theorem \ref{thm:1.4}(ii) for the
filters from the chain.
We would also like to note that the condition of being $n$-permutable for some
$n$ can be also formulated as `having no nontrivial compatible partial order';
this have been attributed to Hagemann, for a~proof see \cite{freese13}.

\begin{theorem} \label{thm:5.1}
  A~variety $\var V$ is not congruence $n$-permutable for any $n$ if and only if
  it contains an algebra compatible with a~partial order that is not an
  antichain. \qed
\end{theorem}

There are two well-known Mal'cev characterizations of the discussed conditions,
the older $(n+1)$-ary terms by Schmidt \cite{schmidt69} and refined ternary
terms by Hagemann and Mitschke \cite{hagemann.mitschke73} (items (2) and (3),
respectively):

\begin{theorem} The following are equivalent for any variety $\var V$ and every
positive integer $n$.
\begin{enumerate}
  \item $\var V$ is $n$-permutable;
  \item there are $(n+1)$-ary $\var V$-terms $s_0,\dots,s_n$ such that
  the identities
    \begin{align*}
      s_0(x_0,\dots,x_n) &\equals x_0,\ s_n(x_0,\dots,x_n) \equals x_n, \\
      s_i(x_0,x_0,x_2,x_2,\dots) &\equals s_{i+1}(x_0,x_0,x_2,x_2,\dots)
        \text{ for odd $i$, and} \\
      s_i(x_0,x_1,x_1,x_3,\dots) &\equals s_{i+1}(x_0,x_1,x_1,x_3,\dots)
        \text{ for even $i$}
    \end{align*}
  are satisfied in $\var V$;
  \item there are ternary $\var V$-terms $p_0,\dots,p_n$ such that
  the identities
    \begin{align*}
      p_0(x,y,z) &\equals x,\ p_n(x,y,z) \equals z,\text{ and} \\
      p_i(x,x,y) &\equals p_{i+1}(x,y,y) \text{ for every $i < n$ }
    \end{align*}
  are satisfied in $\var V$. \qed
\end{enumerate}
\end{theorem}

\subsection{Colorings and linear varieties}

For the characterization of congruence $n$-permutable varieties in the means of
coloring, we will use the smallest structures with two equivalence relations
that do not $n$-permute (denoted $\rel W_n$) for the strong conditions, and the
two-element chain for the general condition.

\begin{definition}
For $n \geq 2$, we define the structure $\rel W_n$ as the relational structure
on $\{ 0,\dots, n \}$ with two binary equivalence relations $\alpha$ and $\beta$
defined by partitions $01|23|\dots$ and $0|12|34|\dots$, respectively.
\end{definition}

The the following is a~variation on a result of Sequeira \cite{sequeira01}, who
provided a~very similar characterization by the means of compatibility with
projections.

\begin{proposition} \label{prop:coloring-for-permutability}
  Let $n\geq 2$. A~variety is not congruence $n$-permutable if and only if has
  strongly $\rel W_n$-colorable terms.
\end{proposition}

\begin{proof}
  First we fix some notation. Let $\rel F$ denote the structure obtained from
  the free structure in $\var V$ generated by $\rel W_n$ by replacing each of
  its relations by its transitive closure. Again, as in the proof of
  Proposition~\ref{prop:coloring-for-modularity}, we obtain that all relations
  of $\rel F$ are equivalences, and $c\colon F \to W_n$ is a~coloring if and
  only if it is a~homomorphism from $\rel F$ to $\rel W_n$. Also, to simplify
  the notation, we identify an~$(n+1)$-ary term $s$ with $s(x_0,\dots,x_n)$.
  
  To prove the first implication, suppose that $\var V$ has Schmidt terms
  $s_0,\dots,s_n$. Our objective is to prove that terms of $\var V$ are not
  strongly $\rel W_n$-colorable. For the contrary, suppose that there is
  a~strong coloring $c\colon \rel F \to \rel W_n$.
  Any such coloring maps $x_0$, and therefore also $s_0$, to $0$. Further, from
  Schmidt identities we get that $(s_1,s_0) \in \beta^{\rel F}$, therefore also
  $c(s_1) = 0$. Following a~similar argument for pairs $(s_{i+1},s_i)$
  alternating between $\beta^{\rel F}$ and $\alpha^{\rel F}$, we obtain $c(s_i)
  < i$ for all $i>0$ which in particular implies that $c(x_n) = c(s_n) < n$,
  a~contradiction with $c(x_n) = n$.

  For the other implication, suppose that $\var V$ does not have Schmidt terms.
  Similarly as in the proof of Proposition~\ref{prop:coloring-for-modularity},
  we define a~coloring $c$ by setting $c(x_0) = 0$ and then, following
  Schmidt-like chains of terms, extend this definition as we are forced, and
  finally, setting $c(f) = n$ for all elements that have not been reached. This
  is summarized in the following definition:
  \[
    c(f) = \min
      \{ i : (x_0,f) \in \freeclo\beta\circ_{i+1} \freeclo\alpha \}
  \]
  if the set on the right hand side is non-empty, and $c(f) = n$, otherwise. By
  definition, we have that for each $i$, $c(x_i) \leq i$.  First, we need to
  show that $c$ is well-defined, in particular that $c(f) \leq n$ for all $f$.
  Suppose that $f$ satisfies $(x_0, f) \in \beta^{\rel F} \join \alpha^{\rel
  F}$. Immediately, we get that $f$ is idempotent, i.e., $f(x_0,\dots,x_0) =
  x_0$. Also, since
  \(
    (x_0,x_i) \in
      \alpha^{\rel F} \circ_{i} \beta^{\rel F} \subseteq
      \beta^{\rel F} \circ_{n+1} \alpha^{\rel F}
  \)
  for all $i$, and $\beta^{\rel F} \circ_{n+1} \alpha^{\rel F}$ is
  a~subuniverse,
  $(x_0,f) \in \beta^{\rel F} \circ_{n+1} \alpha^{\rel F}$.
  Therefore, the aforementioned set contains $n$, and we have $c(f) \leq n$.

  Next, we show that $c$ is compatible with $\alpha$ and $\beta$.  Suppose that
  $(f,g) \in \freeclo \alpha$ and both $f$ and $g$ are idempotent. Without loss
  of generality, suppose that $c(f) \leq c(g)$. If $c(f)$ is an odd number, then
  $c(g) \leq c(f)$; if $c(f)$ is even, then $c(g) \leq c(f) + 1$. Altogether, we
  have that $(c(g),c(f))$ differ by at most 1, and if they do, then $c(f)$ is
  odd and $c(g) = c(f) + 1$. This gives that $(c(f),c(g)) \in \alpha^{\rel
  W_n}$. Similarly, we get that $c$ is compatible with $\beta$, therefore $c$ is
  a~coloring.

  Lastly, we need to show that $c$ is strong. As mentioned before, $(x_0,x_i)
  \in \beta^{\rel F} \circ_{i+1} \alpha^{\rel F}$, therefore $c(x_i) \leq i$ for
  all $i$. For a~contradiction suppose that $c(x_i) < i$ for some $i$, i.e.,
  $(x_0,x_i) \in \beta^{\rel F} \circ_i \alpha^{\rel F}$.
  This implies that there are terms $x_0 = t_0$, $t_1$, \dots, $t_i = x_i$ with
  $(t_j,t_{j+1}) \in \beta^{\rel F}$ for even $j$ and $(t_j,t_{j+1}) \in
  \alpha^{\rel F}$ for odd $j$. Now, define $(i+1)$-ary terms $s_0,\dots,s_i$ by
  putting:
  \[
    s_j (x_0,\dots,x_i) = t_j (x_0,\dots,x_i,x_i,\dots,x_i).
  \]
  We claim that these terms satisfy Schmidt identities. Indeed, $s_0 = x_0$,
  $s_i = x_i$, $s_j(x_0,x_0,x_2,\dots) = s_{j+1}(x_0,x_0,x_2,\dots)$ for odd
  $j$, since $(t_j,t_{j+1}) \in \alpha^{\rel F}$ for such $j$, and
  $s_j(x_0,x_1,x_1,\dots) = s_{j+1}(x_0,x_1,x_1,\dots)$ for even $j$, since
  $(t_j,t_{j+1}) \in \beta^{\rel F}$ for such $j$.
  This gives is a~contradiction with the fact that $\var V$ is
  not even congruence $n$-permutable.
\end{proof}

The corresponding result for the general Mal'cev condition was proven
in~\cite{barto.oprsal.ea15}.

\begin{proposition} [{\cite[Proposition 7.2]{barto.oprsal.ea15}}]
  \label{prop:barto.oprsal.pinsker-n-perm}
  A~variety is not $n$-permutable for any $n$ if and only if it has strongly
  $(\{0,1\},\leq)$-colorable terms. \qed
\end{proposition}

\subsection{Idempotent varieties}

The fact, that the interpretability join of two idempotent varieties that are
not congruence $n$-permutable for any $n$ is not congruence $n$-permutable
either, follows from a~result of Valeriote and Willard
\cite{valeriote.willard14}. No similar result about the corresponding strong
Mal'cev conditions is known.

\begin{theorem}[\cite{valeriote.willard14}] \label{thm:valeriote.willard}
  An~idempotent variety is not $n$-permutable for any $n$ if and only if it is
  interpretable in the variety of distributive lattices. \qed
\end{theorem}

An important step to unify this result with the similar for linear varieties is
the observation that the variety of actions of idempotent polymorphisms of
$(\{0,1\},\leq)$ is interpretable in the variety of distributive lattices and
vice-versa. This is implicitly hidden in \cite{valeriote.willard14}: The key
reasons are that the clone of the two-element lattice on $\{0,1\}$ coincides with
the clone of idempotent polymorphisms of $\leq$, and that the two element
lattice generates the variety of ditributive lattices. As a~direct consequence
of this observation and the above theorem, we get:

\begin{corollary} \label{cor:valeriote.willard}
  An~idempotent variety is not $n$-permutable for any $n$ if and only if it
  contains an~algebra compatible with $(\{0,1\},\leq)$. \qed
\end{corollary}

\subsection{Proofs of Theorems~\ref{thm:overview-introduction}(ii) and
\ref{thm:permutability-introduction}}

First, we address Theorem~\ref{thm:overview-introduction}(ii):
We have showed that if a~variety is not $n$-permutable for any $n$ and it is
either linear, or idempotent then it contains an algebra compatible with
$(\{0,1\},\leq)$ (the linear case follows from
Proposition~\ref{prop:barto.oprsal.pinsker-n-perm} and
Lemma~\ref{lem:linear-varieties}, the idempotent case from Corollary
\ref{cor:valeriote.willard}).  Therefore, given $\var V$ and $\var W$ that are
not $n$-permutable for any $n$ and are both linear, or idempotent, we know that
both contain an algebra that is compatible with $(\{0,1\},\leq)$.
Lemma~\ref{lem:joins} then gives that there is an algebra in $\var V\join \var
W$ that is compatible with this partial order, and therefore witnesses that
$\var V \join \var W$ is not $n$-permutable for any $n$ (see
Theorem~\ref{thm:5.1}).

Theorem~\ref{thm:permutability-introduction} is obtained in a~similar way:
Using Proposition~\ref{prop:coloring-for-permutability} and
Lem\-ma~\ref{lem:linear-varieties}, we get that any linear variety, that is not
$n$-permutable for a~fixed $n$, contains an algebra that is compatible with $\rel
W_n$, therefore also the join of two such varieties contains an algebra
compatible with $\rel W_n$. Having such an~algebra contradicts being
congruence $n$-permutable. \qed
 \section{Satisfying a~non-trivial congruence identity}

There are several Mal'cev characterizations of varieties that satisfy some
non-trivial congruence identity. An older one, called a~Hobby-McKenzie term (for
the definition we refer to \cite{hobby.mckenzie88} or \cite{kearnes.kiss13}),
and more recent by Kearnes and Kiss \cite{kearnes.kiss13}:

\begin{theorem}[{\cite[Theorems 5.28 and 7.15]{kearnes.kiss13}}]
  \label{thm:kearnes.kiss}
The following is equivalent for a~variety $\var V$.
\begin{enumerate}
  \item $\var V$ satisfies a~non-trivial congruence identity,
  \item $\var V$ satisfies an~idempotent Mal'cev condition that fails in the
  variety of semilattices,
  \item there exists $4$-ary terms $t_0,\dots,t_n$ such that the identities
    \begin{align*}
      t_0(x,y,z,w) &\equals x \text{ and } t_n(x,y,z,w) \equals w, \\
      t_i(x,y,y,y) &\equals t_{i+1}(x,y,y,y) \text{ for even $i$,} \\
      t_i(x,x,y,y) &\equals t_{i+1}(x,x,y,y) \text{ and }
      t_i(x,y,y,x) \equals t_{i+1}(x,y,y,x) \text{ for odd $i$}
    \end{align*}
  are satisfied in $\var V$. \qed
\end{enumerate}
\end{theorem}

We will refer to the terms in item (3) as to Kearnes-Kiss terms. Also note that
as a~byproduct of this characterization, Kearnes and Kiss proved that
an~idempotent variety satisfies a~non-trivial congruence identity if and only if
it is not interpretable in the variety of semilattices (this follows from the
equivalence of (1) and (2) in the above theorem).

\subsection{Coloring and linear varieties}

Kearnes and Kiss also proved that item (1) of Theorem \ref{thm:kearnes.kiss}
implies that the variety contains an algebra $\al A$ with a~compatible
(sometimes called commuting) semilattice operation, i.e., there is a~semilattice
operation $\join$ on $A$ such that its graph---the relation $\{ (a,b,a\join b) :
a,b\in A \}$ is a~compatible relation.  This suggests a~relational structure to
be used in the coloring description of this Mal'cev condition: the structure
$\rel S$ on $\{0,1\}$ with one ternary relation $J^{\rel S} = \{ (x,y,x \join y)
: x,y \in \{0,1\} \}$.  Note that the idempotent reduct of $\clo S$ is the same
as the clone of term operations of the semilattice $(\{0,1\}, \join)$
(see also \cite[Lemma~5.25]{kearnes.kiss13}).

\begin{proposition} \label{prop:coloring-congruence-identity}
  A~variety $\var V$ does not have Kearnes-Kiss terms if and only if it has
  strongly $\rel S$-colorable terms.
\end{proposition}

\begin{proof}
  As usual, let $\rel F$ denote the free structure in $\var V$ generated by
  $\rel S$. We will also identify a~binary term $b$ with an element
  $b(x_0,x_1)\in F$.
  Observe that, the relation $J^{\rel F}$ consists of triples $(r,s,t)$ such
  that there exists a~4-ary term $f$ satisfying $f(x,x,y,y) \equals r(x,y)$,
  $f(x,y,x,y) \equals s(x,y)$, and $f(x,y,y,y) \equals t(x,y)$.

  First, suppose that $\var V$ has Kearnes-Kiss terms. We want to proof that
  there is no strong coloring $c\colon \rel F \to \rel S$.
  By the definition of a~coloring and the relation $J$, we can deduce that for
  every triple $(r,s,t) \in J^{\rel F}$ if $c(r) = 0$ and $c(s) = 0$ then also
  $c(t) = 0$, and similarly if $c(t) = 0$ then both $c(r)$ and $c(s)$ are also
  $0$.
  By combining these two observations, one can prove by induction on $i$ that
  for any coloring $c$ with $c(x_0) = 0$ and any an~term $t_i$ from
  Kearnes-Kiss chain we have that $c(t_i(x_0,x_0,x_1,x_1)) = 0$,
  $c(t_i(x_0,x_1,x_0,x_1)) = 0$, and $c(t_i(x_0,x_1,x_1,x_1)) = 0$.
  This shows that any such coloring $c$ has to satisfy $c(x_1) = 0$ which
  contradicts that $c$ is strong.

  For the other implication, suppose that the variety $\var V$ does not have
  Kearnes-Kiss terms. We define a~mapping $c\colon F \to J$ in such a~way that
  $c(t) = 0$ if and only if this fact is forced by the argument in the previous
  paragraph, that is, there exists tuples $(s_i,t_i,r_i)\in J^{\var V}$, $i =
  1,\dots,n$ such that $s_0 = t_0 = x_0$, $r_i = r_{i+1}$ for even $i$, $s_i =
  s_{i+1}$ and $t_i = t_{i+1}$ for odd $i$, and $t = r_n$ for $n$ odd, or $t \in
  \{s_n,t_n\}$ for $n$ even.
  We claim that $c$ is a~coloring. Indeed, if $c(t) = 0$ and $(r,s,t) \in
  J^{\var V}$ for some $r,s\in F$, then also $c(r) = c(s) = 0$ by the
  definition, and if $c(t) = 1$ and $(r,s,t) \in J^{\var V}$ then either $c(r) =
  1$, or $c(s) = 1$, otherwise we would have defined $c(t) = 0$. In either case,
  we have $c(s) \join c(r) = c(t)$ which is what we wanted to prove.

  The coloring $c$ is strong: it maps $x_0$ to $0$ by definition, and also
  $c(x_1) = 1$, since the variety does not have Kearnes-Kiss terms.
\end{proof}

The fact that Kearnes-Kiss terms are not colorable by $\rel S$ can be also
argued using the arguments of \cite{kearnes.kiss13}: having Kearnes-Kiss terms
is an~idempotent linear Mal'cev condition that is not satisfiable in the variety
of semilattices, and existence of a~strong coloring would imply that it is
satisfied in $\clo S$, therefore also in the semilattice $(\{0,1\},\join)$.

\subsection{Idempotent varieties}

There is not much left to prove for idempotent varieties in this case. Let us
just once more reformulate results of Kearnes and Kiss in the following way:

\begin{corollary} \label{cor:kearnes.kiss}
An~idempotent variety does not satisfy a~non-trivial congruence identity if and
only if it contains an~algebra compatible with $\rel S$.
\end{corollary}

\begin{proof}
  Suppose that an~idempotent variety $\var V$ does not satisfy a~non-trivial
  congruence identity. Then by Theorem~\ref{thm:kearnes.kiss}, it is
  interpretable in the variety of semilattices. Therefore, it contains
  an~algebra $\al S$ with $S = \{0,1\}$ whose basic operations are term
  operations of the two-element semilattice $(S; \join)$. But all such term
  operations are compatible with the relation $J^{\rel S}$ which means that $\al
  S$ is compatible with $\rel S$.

  For the other implication, suppose that $\var V$ contains an~algebra $\al S$
  compatible with $\rel S$. If this is the case, one can directly see that the
  terms of $\var V$ are $\rel S$-colorable (define the coloring by taking the
  natural homomorphism from the free algebra generated by $S$ to $\al S$), and
  therefore it does not satisfy a~non-trivial congruence identity by
  Proposition~\ref{prop:coloring-congruence-identity}.
\end{proof}

\subsection{Proof of Theorem \ref{thm:overview-introduction}(iii)}
  The argumentation here follows a~familiar pattern. First, we argue that any
  variety $\var V$ that is either linear, or idempotent has an~algebra
  compatible with $\rel S$: the linear case follows from
  Proposition~\ref{prop:coloring-congruence-identity} and
  Lemma~\ref{lem:linear-varieties}, the idempotent case from
  Corollary~\ref{cor:kearnes.kiss}.

  Next, given that both $\var V$ and $\var W$ do not satisfy a~congruence
  identity and are either linear, or idempotent, we know that both varieties
  contain an algebra compatible with $\rel S$, therefore also $\var V \join \var
  W$ does.  This shows in particular that $\var V \join \var W$ has $\rel
  S$-colorable terms, therefore it does not satisfy a~non-trivial congruence
  identity by Proposition~\ref{prop:coloring-congruence-identity}. \qed
 \section{Having a~cube term}
  \label{sec:cube-terms}

Cube terms describe finite algebras having few subpowers (i.e., with
a~polynomial bound in $n$ on the number of generators of subalgebras of $n$-th
power). This result and many more interesting properties of algebras with cube
terms can be found in \cite{berman.idziak.ea10},
\cite{aichinger.mayr.mckenzie14}, and \cite{kearnes.szendrei12}. There are
several Mal'cev conditions equivalent to having a~cube term, e.g.\ having
an~edge term, or having a~parallelogram term.  For our purpose, the most useful
of these equivalent conditions is the cube term itself.

Fix a~variety, and let $\al F$ be an~algebra that is freely generated by the set
$\{x,y\}$. An~$n$-cube term is a~$(2^n-1)$-ary term $c$ such that
\[
  c^{\al F^n}(\tup x_1,\dots,\tup x_{2^n-1}) = (x,\dots,x)
\]
where $\tup x_1,\dots,\tup x_{2^n-1}$ are all the $n$-tuples of $x$'s and $y$'s
containing at least one $y$. For example, a~Mal'cev term $q$ is
a~$2$-cube term since it satisfies \( q^{\al F^2}((x,y),(y,y),(y,x)) = (x,x) \).
The order of variables in cube terms will not play any role for us.

Although for the main result for cube terms, we only need to study the
corresponding strong Mal'cev conditions, we will continue to formulate all
results also for the general Mal'cev condition of having a~cube term of some
arity.  The reason for that is that the construction we use for the general
Mal'cev condition is a~simpler variant of the constructions we use for the
strong Mal'cev conditions.

Throughout the rest of the section, fix $n\geq 2$.

\subsection{Cube term blockers and crosses}

The notion of cube term blocker was introduced by by Markovi\'c, Mar\'oti, and
McKenzie \cite{markovic.maroti.mckenzie12} to describe finite idempotent
algebras that do not have a~cube term.
We define a~\emph{cube term blocker} to be a~proper subset $U$ of $A$ such that
for every $k\in \mathbb N$, $A^k \setminus (A\setminus U)^k$ is a~subuniverse of
$\al A^k$.

\begin{theorem}[{\cite[Theorem 2.1]{markovic.maroti.mckenzie12}}]
  A~finite idempotent algebra $\al A$ has a~cube term if and only if none of its
  subalgebras has a~cube term blocker. \qed
\end{theorem}

This theorem was recently generalized to idempotent varieties that are not
neccessarily finitely generated by Kearnes and Szendrei
\cite{kearnes.szendrei15}. They also proved a~similar characterization for cube
terms of fixed arity using crosses, that is, relations of the form
\[
  \Cross (U_1,\dots,U_k) =
    \{ (x_1,\dots,x_k): x_i \in U_i \text{ for some $i$}\}.
\]
Note that $U$ is a~cube term blocker if and only if $U$ is a~proper subset of
$A$ and $\Cross(U,\dots,U)$ is a~subuniverse of $\al A^k$ for any $k$. The
result of Kearnes and Szendrei can be formulated as follows.

\begin{theorem}[{\cite[Theorems 2.5 and 3.1]{kearnes.szendrei15}}]
  \label{thm:kearnes.szendrei}
  Suppose that $\var V$ is an~idempotent variety, $\al F$ denotes the
  free algebra generated by the set $\{x,y\}$, and let $n\geq 2$.
  \begin{enumerate}
    \item[(i)] $\var V$ does not have an~$n$-cube term if and only if there exist
      $U_1$, \dots, $U_n \subset F$ such that $x \in U_i$, $y\notin U_i$ for all
      $i$, and \( \operatorname{Cross} (U_1,\dots,U_n) \) is a~subuniverse of
      $\al F^n$.
    \item[(ii)] $\var V$ does not have a~cube term if and only if there exist
      a~cube term blocker $U$ in $\al F$ such that $x \in U$, $y\notin U$.
    \qed
  \end{enumerate}
\end{theorem}

The item (ii) above was also proved by McKenzie and Moore
\cite{mckenzie.moore17}.

Next, we define concrete relational structures that we will use later. Unlike in
\cite{kearnes.szendrei15}, we understand crosses as multi-sorted relational
structures; we encode them as 1-sorted structures whose universes are products
of the original sorts.

\begin{definition} We say that a~structure $\rel C = (C_1\times \dots \times
C_n; \alpha_1,\dots,\alpha_n, R)$ is an~\emph{$n$-cross} if
$\alpha_1$, \dots, $\alpha_n$ are kernels of projections on the corresponding
coordinate, and for each $i \in \{1,\dots,n\}$, there exists $U_i \subseteq C_i$
such that
\[
  R = \{ (a_1,\dots,a_n) \in C_1 \times \dots \times C_n :
    \text{$a_i \in U_i$ for some $i$} \}.
\]
\end{definition}

When defining or talking about $n$-crosses, we will usually only speak about
the relation $R$ since the other relations are implicitly given as kernels of
projections.

\begin{definition} We say that a~relational structure $\rel B = (B; (R_k)_{k\in
\mathbb N})$ is a~\emph{cube term blocker} if there exists $U \subset B$, $U\neq
\emptyset$ such that
\[
  R_k = \{ (a_1,\dots,a_k) \in B^k :
    \text{$a_i \in U$ for some $i$} \}.
\]
\end{definition}

To distinguish between the above definition and the original definition of cube
term blockers, we will always say `an algebra has a~cube term blocker' or `$U$
is a~cube term blocker of an algebra $\al A$' if we refer to the original
definition.
Note that an~algebra $\al A$ has a~cube term blocker $U\subseteq A$ if and
only if it is compatible with some~cube term blocker $\rel B$.

The two transfinite chains of clones will be defined using the following
relational structures:
\begin{definition}
  Let $\kappa$ be an~infinite cardinal, and fix $U_\kappa \subseteq \kappa$ with
  $|U_\kappa| = |\kappa \setminus U_\kappa| = \kappa$. We define a~relation $R_k
  \subseteq \kappa^k$ as
  \[
    R_k = \{ (a_1,\dots,a_k) \in \kappa^k :
      \text{$a_i \in U_\kappa$ for some $i$} \}.
  \]
  The special cube term blocker is defined as $\rel B_\kappa = ( \kappa;
  (R_k)_{k\in \mathbb N} )$, and the special $n$-cross as $\rel C_\kappa = (
  \kappa^n; \alpha_1^{\rel C_\kappa},\dots,\alpha_n^{\rel C_\kappa}, R_n)$ where
  $\alpha_i^{\rel C_\kappa}$ is a~kernel of the $i$-th projection.
\end{definition}

For coloring we will use the smallest cube term blocker $\rel B_0$ and its
reduct $\rel B_n$:
\begin{definition} We define $\rel B_0$ to be the cube term blocker on the set
$\{0,1\}$ defined by $U = \{1\}$, i.e.,
\[
  R_k^{\rel B_0} = \{ 0,1 \}^k \setminus \{ (0,\dots,0) \}
\]
for $k \in \mathbb N$.
We define $\rel B_n$ to be the relational structure $(\{0,1\}; R^{\rel
B_n})$ with $R^{\rel B_n} = R_n^{\rel B_0}$.
\end{definition}

\subsection{Coloring and linear varieties}

We provide characterizations in the means of colorings both for the Mal'cev
condition of having a~cube term of some arity and the strong Mal'cev condition
of having an~$n$-cube term.

\begin{proposition} \label{prop:cube-term-coloring}
The following is true for every variety $\var V$ and all $n\geq 2$.
\begin{enumerate}
  \item[(i)] $\var V$ does not have an~$n$-cube term if and only if it has strongly
    $\rel B_n$-colorable terms.
  \item[(ii)] $\var V$ does not have a~cube term if and only if it has
    strongly $\rel B_0$-colorable terms.
\end{enumerate}
\end{proposition}

\begin{proof}
(i) Let $\rel F_n$ denote the free structure generated by $\rel B_n$.  Observe
that tuples in $R^{\rel F_n}$ are exactly those $n$-tuples that can be the
result of applying some term $f$ coordinatewise to all $n$-tuples consisting of
$x_0$'s and $x_1$'s except the tuple $(x_0,\dots,x_0)$.

To prove the first implication, suppose that $\var V$ does not have a~cube term,
and define a~mapping $c\colon F_n \to B_n$ by:
\[
  c ( f ) = \begin{cases}
    0 & \text{ if $f = x_0$, and } \\
    1 & \text{ otherwise. }
  \end{cases}
\]
Since $\var V$ does not have a~cube term, $R^{\rel F_n}$ does not contain the
tuple $( x_0,\dots, x_0 )$, therefore $c$ is clearly a~coloring from $\rel
F_n$ to $\rel B_n$. It is also strong by definition.

For the other implication, suppose that $\var V$ has an~$n$-cube term $t$.
Therefore, in particular $(x_0,\dots,x_0) \in R^{\rel F_n}$. Any mapping
that maps $x_0$ to $0$ maps this tuple to the tuple $(0,\dots,0)$, therefore
there is no homomorphism $c\colon \rel F_n \to \rel B_n$ with $c(x_0) = 0$, and
in particular no strong coloring of terms of $\var V$.

(ii) Let $\rel F$ denote the free structure generated by $\rel B_0$, and note that
$\rel F$ is an expansion of every $\rel F_k$ (the free structure generated by
$\rel B_k$). Therefore, if a~variety $\var V$ has a~$k$-cube term then its
terms are not strongly $\rel B_0$-colorable since they are not even $\rel
B_k$-colorable. This yields one implication.
For the other implication, suppose that $\var V$ does not have a~cube term. We define
a~mapping $c$ the same way as in the proof of (i). The argument that this is
really a~strong coloring is identical.
\end{proof}

As a~corollary of the above, we get the following.

\begin{corollary} \label{cor:linear-n-cross}
  If $\var V$ is a~linear variety that does not have an $n$-cube term, then it
  contains an~algebra compatible with an $n$-cross of size $2^n$.
\end{corollary}

\begin{proof} From the above and Lemma~\ref{lem:linear-varieties}, we know that
$\var V$ contains an algebra $\al B_n$ compatible with $\rel B_n$. Let $\al C =
\al B_n^n$. We claim that $\al C$ is
compatible with an $n$-cross $\rel C$ with $R^{\rel C} = R^{\rel B_n}$ where
$R^{\rel C}$ is understood as a~unary relation on $C$. Indeed, the operations
of $\al C$ are compatible with kernels of projections since they are defined
component-wise; they are also compatible with $R^{\rel C}$ since component-wise
they act as operations of $\al B_n$ and are therefore compatible with $R^{\rel
B_n}$. Finally, $|C| = 2^n$ since $|B_n| = 2$.
\end{proof}

\subsection{Idempotent varieties}

We encounter a~similar problem as in Section~\ref{sec:modularity.idempotent}: we
are unable to find a~largest idempotent variety without a~cube term (or $n$-cube
term). We will circumvent this problem in a~similar way. The first step is the
following corollary of Theorem~\ref{thm:kearnes.szendrei}.

\begin{corollary} \label{cor:crosses-and-blockers}
  Fix $n\geq 2$ and let $\var V$ be an idempotent variety.
  \begin{enumerate}
    \item[(i)] If $\var V$ does not have an~$n$-cube term, then it contains
      an algebra compatible with an~$n$-cross.
    \item[(ii)] If $\var V$ does not have a~cube term, then it contains
      an algebra compatible with a~cube term blocker.
  \end{enumerate}
\end{corollary}

\begin{proof}
  (i) Suppose that $\var V$ does not have an~$n$-cube term, and let $\al F$
  denote the two-generated free algebra in $\var V$. From
  Theorem~\ref{thm:kearnes.szendrei}(i), we know that there are
  $U_1,\dots,U_n\subseteq F$ such that $\Cross (U_1,\dots,U_n)$ is a~compatible
  relation of $\al F$. We obtain an~algebra compatible with an~$n$-cross by
  repeating the argument from the proof of Corollary~\ref{cor:linear-n-cross}:
  Let $\al A = \al F^n$, and define an~$n$-cross $\rel A$ by taking $R^{\rel A}
  = \Cross^{\al F}(U_1,\dots,U_n)$.

  (ii) This is immediate from Theorem~\ref{thm:kearnes.szendrei}(ii): the
  algebra is $\al F$, and the cube term blocker is
  \(
    \rel F = ( F; \big( F^n \setminus (F \setminus U)^n \big)_{n\in N} ).
  \)
\end{proof}

\begin{proposition}
  \label{prop:crosses-of-arbitrary-size} \label{prop:cube-term-blockers-of-arbitrary-size}
  Let $\var V$ be a~variety and $n\geq 2$.
  \begin{itemize}
    \item[(i)] If $\var V$ contains an algebra compatible with
      an~$n$-cross $\rel C$, then it contains an algebra compatible with $\rel
      C_\kappa$ for all $\kappa \geq \aleph_0 + |\rel C|$
    \item[(ii)] If $\var V$ contains an algebra compatible with a~cube
      term blocker $\rel B$, then it contains an algebra compatible with $\rel
      B_\kappa$ for all $\kappa \geq \aleph_0 + |\rel B|$.
  \end{itemize}
\end{proposition}

\begin{proof}
  Again, we will use Tarski's construction. First, we prove (ii): Suppose that
  $\al B$ is an~algebra in $\var V$ with a~cube term blocker $U$ (recall that
  having a~cube term blocker is equivalent to being compatible with a~relational
  structure which is a~cube term blocker). 
  Put $\al B_0 = \al B$ and $U_0 = U$, and let $\al B_\lambda$ and $U_\lambda$
  denote the algebra and the set obtained by Tarski's construction in $\lambda$
  steps. We claim that $U_\lambda$ is a~cube term blocker of $\al B_\lambda$. We
  will prove that by transfinite induction on $\lambda$: It is true for $\lambda
  = 0$. For ordinal successor $\lambda + 1$, we identify $\al B_{\lambda + 1}$
  with $\al B_\lambda^2$, and $U_{\lambda+1}$ with $B_\lambda \times U_\lambda$.
  We need to prove that $R_k = B_{\lambda+1}^k \setminus
  (B_{\lambda+1} \setminus U_{\lambda+1})^k$ is compatible for all $k\in \mathbb
  N$. But this is true since \(
      B_\lambda^k \setminus ( B_\lambda \setminus U_\lambda )^k
  \) is and 
  \[
    R_k = \{ ((a_1,b_1),\dots,(a_k,b_k)) \in (B_\lambda^2)^k : (b_1,\dots,b_n) \in
      B_\lambda^k \setminus ( B_\lambda \setminus U_\lambda )^k \}.
  \]
  For limit $\lambda$, we get the statement by the standard compactness
  argument.
  In particular, we proved that $U_\lambda$ is a~cube term blocker in $\al
  B_\lambda$ for $|\lambda| = \kappa$.
  For such $\lambda$, we get that the structure $(B_\lambda,
  (R_{k,\lambda})_{k\in \mathbb N})$, where
  $R_{k,\lambda} = B_\lambda^k \setminus ( B_\lambda \setminus U_\lambda)^k$,
  is isomorphic to $\rel B_\kappa$. This is given by the definition of
  $R_{k,\lambda}$ and that $|B_\lambda| = |U_\lambda| = |B_\lambda \setminus
  U_\lambda| = \kappa$. This completes the proof of item (ii).

  The item (i) is obtained by repeating the above argument while
  using Tarski's construction for all $n$ sorts of the $n$-cross. Let
  $\al C$ be the algebra compatible with $\rel C$, let $U_i$ for $i=1,\dots,n$
  be the sets defining the relation $R^{\rel C}$, and let $\al C_1 \times \dots
  \times \al C_n$ be the factorization of $\al C$ so that $\alpha_i^{\rel C}$ is
  the kernel of projection to $\al C_i$. We define algebras $\al C_{i,\lambda}$
  (where $i = 1,\dots, n$ and $\lambda$ ranges through all ordinals) together
  with sets $U_{i,\lambda}$ by Tarski's construction starting with $\al C_{i,0}
  = \al C_i$ and $U_{i,0} = U_i$. By a~similar argument as for item (ii), we
  obtain a~structure of $n$-cross $\rel C'_\lambda$ on the set $C_{1,\lambda}
  \times \dots \times C_{n,\lambda}$ with $R^{\rel C'_\lambda}$ being the
  relation
  \[
    \{ (c_1,\dots,c_n) \in C_{1,\lambda}\times \dots \times C_{n,\lambda} : 
    c_i \in U_{i,\lambda} \text{ for some $i$ } \},
  \]
  and also that this structure is compatible
  with $\al C_{1,\lambda} \times \dots \times \al C_{n,\lambda}$ for all
  $\lambda$. Finally, given that $\lambda$ satisfies $|\lambda| = \kappa$, we claim that
  $\rel C'_\lambda$ is isomorphic to $\rel C_\kappa$. One such isomorphism can
  be defined component-wise, taking for the $i$-th component $a_i$ any bijection
  from $C_{i,\lambda}$ to $\kappa$ that maps $U_{i,\lambda}$ onto $U_\kappa$.
\end{proof}

\begin{lemma} \label{lem:n-crosses-no-terms}
  If a~variety contains an algebra compatible with an $n$-cross, it does not
  have an~$n$-cube term.
\end{lemma}

\begin{proof}
  For the contrary, suppose that $t$ is a~cube term and let $\rel C$ be
  an~$n$-cross. We claim that no algebra $\al C$ in the variety is compatible
  with $\rel C$. To prove that, suppose that $\al C$ is an algebra on $C$ whose
  operations are compatible with $\alpha_i^{\rel C}$ for $i = 1,\dots,n$.
  Therefore, $\al C$ factors as $\al C_1 \times \dots \times \al C_n$ with
  $\alpha_i^{\al C}$ being the kernels of projections.  Let $U_1\subset C_1$,
  \dots, $U_n\subset C_n$ be such that
  \[
    R^{\rel C} = \{ (c_1,\dots,c_n) \in C_1\times \dots \times C_n :
      c_i \in U_i \text{ for some $i$ } \},
  \]
  and choose $a_i\in U_i$ and $b_i\in C_i \setminus U_i$.
  If we apply $t^{\al C}$ on the tuples from the set
  \[
    \{ \{a_1,b_1\} \times \dots \times \{a_n,b_n\} \} \setminus
    \{(b_1,\dots,b_n)\} \subseteq R^{\rel C}
  ,\]
  we obtain $(b_1,\dots,b_n) \notin R^{\rel C}$. This is due to the fact that
  $t^{\al C}$ acts coordinate-wise as $t^{\al C_i}$ and $t$ is a~cube term of
  the variety. This concludes that $t^{\al C}$ is not compatible with $R^{\rel
  C}$, and consequently $\al C$ is not compatible with $\rel C$.
\end{proof}

\subsection{Proof of Theorem~\ref{thm:overview-introduction}(iv)}

We have proven that, given a~variety $\var V$ that is either linear, or
idempotent and does not have an $n$-cube term, $\var V$ has an~algebra
compatible with an~$n$-cross $\rel C^{\var V}$ (see
Corollaries~\ref{cor:linear-n-cross} and \ref{cor:crosses-and-blockers}(i)).
Given two such varieties $\var V$ and $\var W$, and $n$-crosses $\rel C^{\var
V}$ and $\rel C^{\var W}$, we know from
Proposition~\ref{prop:crosses-of-arbitrary-size}(i) that both $\var V$ and $\var
W$ contain an~algebra compatible with the $n$-cross $\rel C_\kappa$ for all
infinite cardinals $\kappa \geq |C^{\var V}| + |C^{\var W}|$. Therefore, by
Lemma~\ref{lem:joins} also their interpretability join does. Finally,
Lemma~\ref{lem:n-crosses-no-terms} gives that $\var V \join \var W$ does not
have an~$n$-cube term. \qed

\subsection{Cofinal chains}

We will discuss some properties of the transfinite sequences
\(
  \clo B_0, \clo B_{\aleph_0}, \clo B_{\aleph_1}, \dots \) and \(
  \clo C_{\aleph_0}, \clo C_{\aleph_1}, \dots
\)
where $\clo B_\kappa$ are polymorphism clones of cube term blockers $\rel
B_\kappa$, and $\clo C_\kappa$ are polymorphism clones of the $n$-crosses $\rel
C_\kappa$ for a~fixed $n\geq 2$.
We show that these sequences form strictly increasing chains in the lattice of
clones, and as a~corollary thereof, we obtain that there is no maximal (in the
interpretability order) idempotent variety without an $n$-cube term.

\begin{corollary} Let $\lambda < \kappa$ be two infinite cardinals, and $n\geq
2$. Then
  \begin{enumerate}
    \item[(i)] there is a~clone homomorphism from $\clo C_\lambda$ to $\clo
    C_\kappa$, and 
    \item[(ii)] there is a~clone homomorphism from $\clo B_\lambda$ to $\clo B_\kappa$.
  \end{enumerate}
\end{corollary}

\begin{proof}
  The statement follows from Proposition~\ref{prop:crosses-of-arbitrary-size} in
  the same way as Corollary~\ref{cor:modularity-chain} follows from
  Proposition~\ref{prop:huge-pentagons}.
\end{proof}

In order to prove that there is no largest idempotent variety without a~cube
term, we need to restrict to idempotent reducts of the discussed clones.

\begin{lemma}
  \label{lem:cube-term-large-algebras} \label{lem:previous-proposition}
  Let $\kappa$ be an infinite cardinal. The idempotent reduct of
  $\clo B_\kappa$ has no action on any set $C$ with $1 < | C | < \kappa$.
\end{lemma}

\begin{proof}
Let $\idclo B_\kappa$ denote the idempotent reduct of $\clo B$. We will show
that if $\clo B_\kappa$ has an~action on a~set $C$ of size at least $2$ then
$|C| \geq \kappa$.

To prove that we find identities that are satisfied in $\idclo B_\kappa$ but
cannot be satisfied by functions on a~set of size smaller then $\kappa$ unless
this set has only one element. These identities are very similar to those used
in Proposition~\ref{prop:modularity-large-algebras}. They use binary
symbols $f_i$, $i\in \kappa$ and ternary symbols $p_{i,j}$, $q_{i,j}$, $i,j \in
\kappa$:
\begin{align*}
  x &\equals p_{i,j}(x,f_j(x,y),y), \\
  p_{i,j}(x,f_i(x,y),y) &\equals q_{i,j}(x,f_j(x,y),y), \\
  q_{i,j}(x,f_i(x,y),y) &\equals y
\end{align*}
for all $i\neq j$, and $f_i(x,x) \equals x$ for all $i$. Note that these
identities force that all the functions are idempotent.

We claim that $\rel B_\kappa$ has polymorphisms satisfying these equations:
Fix $c\in U_\kappa$, put $f_i (x,y) = i$ for $x
\neq y$, $f_i(x,x) = x$, and define $p_{i,j}$ and $q_{i,j}$ by:
\begin{align*}
  p_{i,j} (x,y,z) &= \begin{cases}
    x & \text{ if $y = f_j(x,z)$, and } \\
    c & \text{ otherwise; }
  \end{cases} \\
  q_{i,j} (x,y,z) &= \begin{cases}
    z & \text{ if $y = f_i(x,z)$, and } \\
    c & \text{ otherwise. }
  \end{cases}
\end{align*}
Clearly, these operations satisfy the required identities. To prove that they
are also compatible with $\rel B_\kappa$, observe that any operation $t$ which
has a~coordinate $i$ such that \( t(x_1,\dots,x_n) \in U_\kappa \) whenever $x_i
\in U_\kappa$ is a~polymorphism of $\rel B_\kappa$.
The corresponding coordinates for our functions are: the first for $p$'s, the
last for $q$'s, and arbitrary for~$f$'s.

The last part is to prove that these identities are not satisfiable in any
non-trivial algebra of size strictly less than $\kappa$. This follows from the
same argument as Proposition~\ref{prop:modularity-large-algebras}. In fact,
these identities are a~stronger version of those used in the mentioned proof.
\end{proof}

\begin{corollary} Let $\lambda < \kappa$ be two infinite cardinals, $n\geq
2$, and let $\idclo A$ denote the idempotent reduct of a~clone $\clo A$. Then
  \begin{enumerate}
    \item[(i)] there is no clone homomorphisms from $\idclo C_\kappa$ to $\idclo
    C_\lambda$, and 
    \item[(ii)] there is no clone homomorphisms from $\idclo B_\kappa$ to
    $\idclo B_\lambda$.
  \end{enumerate}
\end{corollary}

\begin{proof}
  The item (ii) follows directly from the above lemma. To prove item (i),
  we first observe that there is a~clone homomorphism $\xi$ from $\idclo B_\kappa$ to
  $\idclo C_\kappa$: define $\xi(f)$ to be the component-wise action of $f$ on
  $\kappa^n$. Indeed, if $f\in \idclo B_\kappa$, $\xi(f)$ preserves $R^{\rel
  C_\kappa}$ since it preserves $R_n^{\rel B_\kappa}$. Now, if there was
  a~clone homomorphism from $\idclo C_\kappa$ to $\idclo C_\lambda$, we would
  get one from $\idclo B_\kappa$ to $\idclo C_\lambda$ by composing it with
  $\xi$ which contradicts Lemma~\ref{lem:previous-proposition}.
\end{proof}

By analogous proofs as \ref{cor:no-strongest-non-modular-variety} we obtain the
following.

\begin{corollary} \label{cor:no-strongest-variety-without-cube-term}
Fix $n\geq 2$.
\begin{enumerate}
  \item[(i)] The class of all interpretability classes of idempotent varieties
    that do not have an~$n$-cube term does not have a~largest element.
  \item[(ii)] The class of all interpretability classes of idempotent varieties
    that do not have a~cube term does not have a~largest element. \qed
\end{enumerate}
\end{corollary}

\subsection{Remark(s)} The presented proof of Theorem \ref{thm:1.4} was
developed independently of Kearnes and Szendrei after their announcement of
Theorem~\ref{thm:kearnes.szendrei}, and before the final manuscript
\cite{kearnes.szendrei15} was available.
 
\subsection*{Acknowledgements}

The results of the present paper are a~part of the author's Doctoral Thesis. The
author would like to express his gratitude to his advisor, Libor Barto, for his
supervision of this work.

\newcommand{\etalchar}[1]{$^{#1}$}

\end{document}